\documentclass[a4paper,11pt,leqno]{article}

\usepackage[utf8]{inputenc}
\usepackage[T1]{fontenc} 
\usepackage[english]{babel}
\usepackage{verbatim}
\usepackage{calligra}
\usepackage{enumerate}
\usepackage{dsfont}
\usepackage{amsfonts}
\usepackage{amsmath}
\usepackage{amsthm}
\usepackage{amssymb}
\usepackage{mathtools}
\usepackage{mathrsfs}
\usepackage{color}
\usepackage{graphicx}
\usepackage{bm}
  \usepackage[all]{xy}
\usepackage{hyperref}
\usepackage{comment}
\hypersetup{
    colorlinks=true,
    linkcolor=black,
    filecolor=black,      
    urlcolor=black,
}

\usepackage{tikz}

\usepackage{graphicx}		
\usepackage[hypcap=false]{caption}

\DeclareMathOperator*{\tend}{\longrightarrow}

\DeclareMathOperator*{\egfty}{=}

\DeclareMathOperator*{\D}{\rm{div}}

\DeclareMathOperator*{\limss}{\overline{\rm lim}}

\theoremstyle{plain}
\newtheorem{defi}{Definition}
\newtheorem{thm}[defi]{Theorem}
\newtheorem{prop}[defi]{Proposition}
\newtheorem{cor}[defi]{Corollary}
\newtheorem{lemma}[defi]{Lemma}

\theoremstyle{definition}

\newtheorem{rmk}[defi]{Remark}

\newcommand{\mc}{\mathcal}

\newcommand{\what}{\widehat}
\newcommand{\wtilde}{\widetilde}

\renewcommand{\t}{\tau}

\newcommand{\R}{\mathbb{R}}

\newcommand{\Z}{\mathbb{Z}}

\renewcommand{\P}{\mathbb{P}}

\newcommand{\curl}{{\rm curl}\,}

\newcommand{\dx}{ \, {\rm d} x}
\newcommand{\dt}{ \, {\rm d} t}

\newcommand{\dy}{ \, {\rm d} y}

\allowdisplaybreaks

\begin{document}

\newcommand{\cobb}[1]{\textcolor{red}{[***DC: #1 ***]}}
\newcommand{\koch}[1]{\textcolor{blue}{[***HK: #1 ***]}}

\title{\textsc{\Large{\textbf{Unbounded Yudovich Solutions of the Euler Equations}}}}

\author{Dimitri Cobb and Herbert Koch \vspace{0.5cm}\\
\footnotesize{\textsc{Universität Bonn}} -- {\footnotesize \it Mathematisches Institut} \vspace{.1cm} \\
\footnotesize{\ttfamily{cobb@math.uni-bonn.de}}\\
\footnotesize{\ttfamily{koch@math.uni-bonn.de}}
}

\vspace{.2cm}
\date\today

\maketitle

\begin{abstract}
    In this article, we will study unbounded solutions of the 2D incompressible Euler equations. One of the motivating factors for this is that the usual functional framework for the Euler equations (\textsl{e.g.} based on finite energy conditions, such as $L^2$) does not respect some of the symmetries of the problem, such as Galileo invariance. 

    Our main result, global existence and uniqueness of solutions for initial data with square-root growth $O(|x|^{\frac{1}{2} - \epsilon})$ and bounded vorticity, is based on two key ingredients. Firstly an integral decomposition of the pressure, and secondly examining local energy balance leading to solution estimates in local Morrey type spaces. We also prove continuity of the initial data to solution map by a substantial adaptation of Yudovich's uniqueness argument.
\end{abstract}


\smallskip

{\footnotesize \textbf{Mathematics Subject Classification (MSC2020):} 
35Q31 (main), 
76B03, 
35Q35, 
35S30, 
46E30 (secondary). 
}

\section{Introduction}

The purpose of this article is to study the existence and uniqueness of unbounded solutions of the 2D Euler equations with Yudovich regularity. These equations describe the evolution of a perfect incompressible fluid, and are written as follow:
\begin{equation}\label{ieq:Euler}
    \begin{cases}
    \partial_t u + \D (u \otimes u) + \nabla \pi = 0 \\
    \D(u) = 0.
    \end{cases}
\end{equation}
In the above, $u(t,x) \in \R^d$ is the fluid velocity at time $t \in \R$ and point $x \in \R^d$, while $\pi(t,x) \in \R$ is the pressure field. The divergence operator $\D$ is defined on vector fields $u = (u_k)_{1 \leq k \leq d}$ and matrix fields $A = (A_{ij})_{1 \leq i, j \leq d}$ by, respectively, $\D(u) = \sum_k \partial_k u_k$ and $\D(A) = \sum_i \partial_i A_{ij} e_j$, where $(e_j)_{1 \leq j \leq d}$ is the canonical basis of $\R^d$. In most of this article, we focus on plane fluids $d=2$, although some of our results will also have applications to general dimensions $d \geq 2$. More precisely, we will show existence and uniqueness of solutions for initial data with bounded vorticity and at most square root growth:
\begin{equation}\label{ieq:UnboundedVelocity}
    \big| u_0(t,x) \big| \egfty_{|x| \rightarrow \infty} O \left( |x|^\alpha \right) \qquad \text{and} \qquad \| \omega_0 \|_{L^\infty(\R^2)} < \infty
\end{equation}
for some $0 \leq \alpha < \frac{1}{2}$, where $\omega_0 = \curl(u_0)$. As we do so, we will also have to examine some qualitative properties of unbounded solutions, such as the distribution of kinetic energy within the fluid. We will be exclusively using Eulerian methods (as opposed to Lagrangian methods based on flow maps), and require no decay of any kind on the vorticity.

\subsection{Presentation of the Problem}

There are many reasons to study unbounded solutions of the Euler equations. At first, this may seem surprising, as the conservation of energy implies that (regular enough) solutions should have finite total kinetic energy, and hence lie in the space $L^\infty_t(L^2_x)$. However, equations \eqref{ieq:Euler} are also invariant under the operation of the Galilean group: if $u(t,x)$ is a solution and $V \in \R^d$ is a fixed velocity defining a change of inertial reference frames, then
\begin{equation*}
    v(t,x) = u \big( t, x - Vt \big) + V
\end{equation*}
also is a solution of \eqref{ieq:Euler}. And if $u$ has finite energy, then it will certainly not be the case for $v$. A more natural functional setting would then be to study solutions that lie in a space that is stable under Galileo transform, unlike $L^\infty_t(L^2_x)$.

A second reason why the space $L^\infty_t(L^2_x)$ is inadequate is that solutions of \eqref{ieq:Euler} have no chance of being unique unless some regularity assumption is made. Typically, essentially all uniqueness results so far require the solution $u$ to be at least (log-)Lipschitz\footnote{To be precise, the flow of $u$ has to satisfy some Osgood type property. In this article, we will work with log-Lipschitz solutions.} (although actually \textsl{proving} a non-uniqueness result is extremely challenging, see for example \cite{Vishik1, Vishik2} for the case of a forced system with initial data $\omega_0 \in L^p$ and $p < \infty$, or \cite{BC, BM} for non-uniqueness of the non-forced equations \eqref{ieq:Euler} for initial data $\omega_0$ in a Lorentz space $L^{1, \infty}$, in the Hardy spaces $\mc H^p$ with $p < 1$, or the recent preprint \cite{BCK} for the case of $\omega_0 \in L^{1 + \epsilon}$). This (log)-Lipschitz requirement is compatible with the scaling invariance of the equations: if $\lambda, \mu > 0$ and $u(t,x)$ is a solution of the Euler equations, then
\begin{equation*}
    u_{\mu \lambda}(t,x) := \mu \lambda^{-1} u \big( \mu t, \lambda x \big)
\end{equation*}
also is a solution, so that more natural function spaces should share the scaling of $L^1_t(\dot{W}^{1, \infty}_x)$. Unfortunately, this is not the case for $L^\infty_t(L^2_x)$. However, as a set of initial data, $\dot{W}^{1, \infty}$ suffers from severe ill-posedness in the Euler  equations: consider for example the solution (see \cite{CW, Miller})
\begin{equation*}
    u(t,x) = \frac{1}{ 1 - t/T} \left(
    \begin{array}{c}
        x_1  \\
        -x_2 
    \end{array}
    \right) \qquad \text{and} \qquad \pi(t,x) = - \frac{1}{2T(1 - t/T)^2} \big( x_1^2 - x_2^2 + T (x_1^2 + x_2^2) \big),
\end{equation*}
which blows-up at time $T > 0$ that can be chosen arbitrarily small while keeping the initial velocity $u_0(x) = (x_1, -x_2)$ in the space $\dot{W}^{1, \infty}$. While this issue can be resolved by prescribing an asymptotic behavior to the pressure force \cite{CW} or symmetry conditions on the problem \cite{Elgindi2020, Miller}, we will not study it in this paper, as we restrict our attention to solutions which grow at most as $O(|x|^{\frac{1}{2} - \epsilon})$, and so cannot exhibit the pathological behavior described here. More general solutions will be the object of a forthcoming paper. However, the comments above highlight the importance of studying unbounded solutions.

Finally, from a very concrete point of view, infinite energy solutions can be used to represent large scale flows, such as geophysical fluids, where far-away boundary conditions (\textsl{e.g.} coastlines) are secondary to the dynamics. For example, in the framework of the $\beta$-plane approximation, variations of latitude are neglected, so that the computations are only valid in a mid-latitude region of a few thousand kilometers wide. Omission of coastlines outside of the region in consideration should then be seen as of little importance compared to the much rougher $\beta$-plane approximation. Another example where unbounded solutions of the form \eqref{ieq:UnboundedVelocity} may be physically relevant is the study of phenomena such as hurricanes where the velocity increases away from a central point. 

\medskip

In any event, unbounded solutions pose their own interesting questions and challenges. The first one is the determination of appropriate far-field conditions to insure uniqueness of (regular) solutions. Indeed, if $(u, \pi)$ is a solution of \eqref{ieq:Euler} with initial datum $u_0$, then it is possible to construct an infinite family of solutions that share the same initial datum by means of a non-inertial change of reference frames. Consider $g \in C^\infty (\R ; \R^d)$ such that $g(0) = 0$, and the functions $(v, p)$ defined by
\begin{equation}\label{ieq:GenGalileo}
    v(t,x) = u\big( t, x - G(t) \big) + g(t) \qquad \text{and} \qquad p(t,x) = \pi \big( t, x - G(t) \big) - g'(t) \cdot x,
\end{equation}
where $G(t) = \int_0^t g$ is the primitive of $G$ that cancels at time $t=0$. Then the couple $(v, p)$ also is a solution of the Euler equations with initial datum $u_0$. The operation \eqref{ieq:GenGalileo} is sometimes referred to as a generalized Galilean transform, and can also be understood as the application of a pressure gradient  ``from infinity''. At any rate, the operations \eqref{ieq:GenGalileo} compromise uniqueness unless some kind of far-field condition is prescribed on the solution. In this work, we will study solutions $u$ such that
\begin{equation}\label{ieq:FarFieldCondition}
    u(t) - u(0) \in \mc S'_h,
\end{equation}
where $\mc S'_h$ is Chemin's space of tempered distributions $f \in \mc S'$ whose Fourier transform satisfy a weak vanishing property at low-frequencies: if $\chi \in \mc D$ is a cut-off function with $\chi(\xi) = 1$ for $|\xi| \leq 1$, then $f \in \mc S'_h$ if and only if
\begin{equation}\label{ieq:CheminSpace}
    \chi(\lambda \xi) \what{f}(\xi) \tend_{\lambda \rightarrow \infty} 0 \qquad \text{in } \mc S'.
\end{equation}
In terms of physical variables, this means that the function $f \in \mc S'_h$ has a (weak) zero averaging property at large scales. For example, the sign function $\sigma = \mathds{1}_{\R_+} - \mathds{1}_{\R-}$ is an element of $\mc S'_h$, while non-zero polynomial functions are not elements of $\mc S'_h$. We refer to \cite{Cobb1, Cobb2} for a discussion of \eqref{ieq:FarFieldCondition} as well as its relation to other types of far-field conditions, and how uniqueness of regular solutions can be deduced. Using the far-field condition \eqref{ieq:FarFieldCondition}, we will prove the existence and uniqueness of 2D solutions which possess a Morrey-Yudovich type regularity. More precisely, we will show well-posedness in the class of divergence-free velocities $u : \R_+ \times \R^2 \longrightarrow \R^2$ such that $\omega \in L^\infty (\R_+ \times \R^2)$ and
\begin{equation}\label{ieq:YudovichSpace}
    \| u(t) \|_{L^2_\alpha}^2 := \sup_{R \geq 1} \left[ \frac{1}{R^{2 + 2 \alpha}} \int_{|x|\leq R} \big| u(t,x) \big|^2 \dx \right] \in L^\infty_{\rm loc} (t \in \R_+),
\end{equation}
for $0 \leq \alpha < \frac{1}{2}$. The space $L^2_\alpha$ is a \textsl{non-homogeneous local Morrey space} (see Definition \ref{d:MorreySpace} and Remark \ref{r:Morrey} below), but we will colloquially refer to it as a \textsl{Morrey space}. Note that the $L^2_\alpha$ condition is consistent with $O(|x|^\alpha)$ growth of the solution. More comments on the result and its proof can be found below.


Another problem which will turn out to be crucial in this paper is the study of the distribution of kinetic energy inside the solution. This should be thought of as an attempt to understand the interaction between large and small scales in an incompressible fluid. More precisely, in an infinite energy and incompressible fluid, it may be feared that non-local behavior may lead to energy concentration phenomena. If $B \subset \R^2$ is a given ball, we may wonder how the quantity $E_B = \| u \|_{L^2(B)}$ evolves with respect to time for solutions of the type \eqref{ieq:YudovichSpace}. We will prove that $E_B$ can grow at most at algebraic speed
\begin{equation*}
    E_B(t) = O \left( t^{2 \frac{1 + \alpha}{1 - \alpha}} \right) \qquad \text{as } t \rightarrow \infty.
\end{equation*}
The fact that this estimate is \textsl{global in time} will be essential to proving our result. We refer to the helpful survey \cite{G} for an introduction to these questions.

\subsection{Previous and Related Results}

As the reader will probably have guessed from the paragraphs above, there already is a considerable literature linked to the study of infinite energy or unbounded solutions in incompressible hydrodynamics, with some of the first results dating from as early as 1979 \cite{Cantor1979}.

\subsubsection{Existence and Uniqueness of Infinite Energy and Unbounded Solutions}

Existence and uniqueness of such solutions have first been established by M. Cantor \cite{Cantor1979} in weighted spaces by using flow-map (\textsl{i.e.} Lagrangian) techniques, inspired by the seminal work of Ebin and Marsden \cite{EbinMardsen}. In an important article, Benedetto, Marchioro and Pulvirenti \cite{BMP} were able to construct unique unbounded solutions $u$ of the 2D Euler problem such that
\begin{equation*}
    u(t,x) = O \big( |x|^\alpha \big) \qquad \text{and} \qquad \omega \in L^p \cap L^\infty,
\end{equation*}
with $0 < \alpha < 1$ and $p < 2/\alpha$. The case where $\alpha = 0$ and $p = \infty$ was reached by Serfati \cite{Serfati} by using an integral decomposition of the Biot-Savart kernel comparable to the one we use. We also refer to \cite{AKLFNL} for details on Serfati's argument. Serfati's ideas have also been used by Ambrose, Cozzi, Erickson and Kelliher \cite{ACEK} to prove existence and uniqueness of solutions for SQG and 3D Euler in uniformly local Sobolev spaces. Brunelli proved \cite{Brunelli} existence and uniqueness of solutions when $\alpha = 1/2$ and $p = \infty$ under an integrability assumption on the vorticity, namely
\begin{equation*}
    \int \frac{|\omega (x)|}{1 + |x|} \dx < \infty.
\end{equation*}
Still in the case $d = 2$, the setting $\alpha < 1/2$ and $p = \infty$ was later revisited by Cozzi and Kelliher \cite{CK} who applied flow-map techniques to Serfati's argument to recover local in time well-posedness without any integrability assumption on $\omega$, as well as a global in time result under much strong restrictions on the growth on the velocity (say $u(t,x) = O \big( \log |x| \big)$ for example). This paper followed a previous result \cite{Cozzi} of Cozzi showing well-posedness of Yudovich solutions when the velocity field posesses mild (logarithmic) growth. Lastly concerning the 2D setting, Chepyzhov and Zelik \cite{CZ} generalized Yudovich's proof of uniqueness to bounded solutions.

\medskip

In the general case of dimensions $d \geq 2$, and beyond the paper \cite{Cantor1979} of M. Cantor, local unique solutions of \eqref{ieq:Euler} are constructed by Pak and Park \cite{PP} in the Besov-Lipschitz space $B^1_{\infty, 1} \subset W^{1, \infty}$, and by McOwen and Topalov \cite{OT} in weighted Sobolev spaces assuming the solutions grows at most as $O\big( |x|^{\frac{1}{2} - \epsilon} \big)$. Solutions with linear growth are examined in \cite{Elgindi2020} by Elgindi and Jeon under symmetry assumptions, in order to avoid the ill-posedness issue raised above. Chae and Wolf \cite{CW} prove local well-posedness in a Campanato type space $\mc L^1_{1, (1)}$, which includes functions $u_0$ that grow as $u_0(x) = O(|x|)$, by prescribing the asymptotic behavior of the pressure force. This last space is at the center of the chain of embeddings $B^1_{\infty, 1} / \R \subset \dot{B}^1_{\infty, 1} \subset \mc L^1_{1, (1)} \subset \dot{W}^{1, \infty}$, so that the result in \cite{CW} improves that of \cite{PP} while still requiring Lipschitz regularity. We should also mention the paper \cite{ST} of Sun and Topalov, which examines existence and uniqueness of quasi-periodic solutions.

\subsubsection{The Role of Far-Field Conditions in Uniqueness}\label{ss:PressureIntro}

All the well-posedness results above rely on a way to represent the pressure as an integral operator. Formally, the pressure $\pi$ solves the elliptic equation
\begin{equation}\label{ieq:PressurePoisson}
    - \Delta \pi = \sum_{j, k} \partial_j u_k \partial_k u_j,
\end{equation}
which is obtained by taking the divergence of the first equation in \eqref{ieq:Euler}. In particular, this means that $\pi$ is only determined by \eqref{ieq:PressurePoisson} up to the addition of a harmonic function. In the absence of further constraints (such as far-field conditions), this undetermination leads to loss of uniqueness of solutions for problem \eqref{ieq:Euler}. Kukavica (see \cite{Kukavica} and the later article \cite{KV} of Kukavica and Vicol) has shown, in the very similar context of Navier-Stokes equations, that this ill-posedness issue is related to the invariance of the problem under the non-inertial change of reference frames \eqref{ieq:GenGalileo}. A natural question is then to understand what type of far-field condition should be required for the Euler equations to be well-posed.

This problem has mainly been investigated for the Navier-Stokes equations, in the context of bounded solutions, in the past twenty years or so. In \cite{GIKM}, Giga, Inui, J. Kato and Matsui have proved that (regular enough) bounded solutions are uniquely determined under the condition that the pressure takes the form
\begin{equation*}
    \pi = \pi_0 + \sum_{j, k} R_j R_k \pi_{j, k},
\end{equation*}
where $\pi_0(t), \pi_{j, k}(t) \in L^\infty$ and $R_j = \partial_j (- \Delta)^{-1/2}$ are the Riesz transforms. J. Kato \cite{KatoJ} later relaxed this condition by simply requiring that the pressure have bounded mean oscillations $\pi(t) \in {\rm BMO}$. These conditions allow the pressure force $- \nabla \pi$ to be uniquely represented by an integral formula (typically a derivative of Riesz transforms $\nabla \pi = \nabla (- \Delta)^{-1} \nabla^2 : (u \otimes u)$), and lead to uniqueness of regular enough solutions. In \cite{KV}, Kukavica and Vicol show that the condition
\begin{equation}\label{ieq:KVCondition}
    \pi(t, x) = o (|x|) \qquad \text{as } |x| \rightarrow \infty
\end{equation}
is enough to ensure uniqueness of solutions. Other results involving the pressure are finite moment conditions (see the articles \cite{Maremonti} of Maremonti, or \cite{AsAsFd} of \'Alavarez-Samaniego, \'Alavarez-Samaniego and Fern\'andez-Dalgo) or a Campanato type condition (see the work \cite{NY} of Nakai and Yoneda).

All these results use a far-field condition on the pressure to obtain uniqueness of solutions, but it is also possible to replace these by conditions involving the velocity field. For example, Fern\'andez-Dalgo and Lemarié-Rieusset \cite{FL} show that it is enough for the kinetic energy $\frac{1}{2}|u|^2$ to satisfy the following weighted integrability condition
\begin{equation}\label{ieq:FdLrCondition}
    \forall T > 0, \qquad \int_0^T \int \frac{|u(t,x)|^2}{(1 + |x|)^d} \dx \dt < \infty,
\end{equation}
and likewise, Lemarié-Rieusset, in point \textit{ii} of Theorem 11.1 pp. 109--111 of the book \cite{LM}, imposes a Morrey type condition
\begin{equation}\label{ieq:LmCondition}
    \forall t_1 < t_2, \qquad \sup_{x \in \R^d} \frac{1}{\lambda^d} \int_{t_1}^{t_2} \int_{|x - y| \leq \lambda} |u(t,x)|^2 \dx \dt \tend_{\lambda \rightarrow \infty} 0.
\end{equation}
It turns out that condition \eqref{ieq:LmCondition} is weaker than \eqref{ieq:FdLrCondition}, see Lemma 29 in the first author's paper \cite{Cobb1}. In \cite{CF3}, F. Fanelli and the first author discuss how mild integrability ($L^p$ for $p < \infty$) can be used to obtain uniqueness in the ideal MHD equations. The paper \cite{CW} of Chae and Wolf is quite unique in that it obtains uniqueness of solutions with linear growth based on a Morrey type decay of the pressure force and the assumption that the velocity field $u(t,x)$ is centered, in other words that
\begin{equation*}
    \forall t, \qquad u(t, 0) = 0.
\end{equation*}
Of course, this implies that the initial velocity is zero at the origin $u_0(0) = 0$, but this is a harmless assumption thanks to the Galilean invariance of the equations. Unfortunately, we will not be able to cover the case of linear growth in this work (a followup paper is planned), as square root growth $O(|x|^{\frac{1}{2} - \epsilon})$ is needed to enforce the far-field condition \eqref{ieq:FarFieldCondition}. We also should mention the paper \cite{Kelliher} of Kelliher dedicated to the 2D setting.

\medskip

Finally, a unified perspective on this problem is provided in the paper \cite{Cobb1} of the first author, where it is proven that uniqueness of (regular) bounded solutions $u \in C^0(L^\infty)$ follows from either conditions
\begin{enumerate}[(i)]
  \setlength{\itemsep}{0.5pt}
  \setlength{\parskip}{0pt}
  \setlength{\parsep}{0pt}
    \item $u(t) - u(0) \in \mc S'_h$, where $\mc S'_h$ is the space of tempered distributions satisfying \eqref{ieq:CheminSpace} ;
    \item $\pi(t) \in {\rm BMO}$ ;
    \item $\pi(t,x) = O \big( \log |x| \big)$ for $|x| \rightarrow \infty$;
    \item $\pi(t,x) = o (|x|)$ for $|x| \rightarrow \infty$.
\end{enumerate}
In addition, these four conditions are equivalent, and are in some sense optimal. This shows that many of the conditions described above were in fact equivalent, for instance the one $\pi(t) \in {\rm BMO}$ of J. Kato \cite{KatoJ} and the one \eqref{ieq:KVCondition} of Kukavica and Vicol. We refer to \cite{Cobb1} for further discussion on the different conditions and how they are related.

\subsubsection{Distribution of Kinetic Energy in Solutions}

To finish this short overview of the literature on 2D infinite energy solutions, we cover a few previous works that are related to understanding the distribution of kinetic energy in the fluid. As we have explained, this topic and the methods developped to explore it are essential to our purposes. A very helpful survey of that topic by Gallay can be found in \cite{G}. As we mentioned above, the question is whether the non-local behavior of a 2D incompressible fluid leads to concentration of energy in a given ball $B \subset \R^2$. This involves examining the behavior of the quantity $E_B(t) = \| u(t) \|_{L^2(B)}$ at large times. Zelik shows \cite{Z} that bounded 2D solutions $u$ with bounded vorticity $\omega$ fulfill the estimate
\begin{equation*}
    E_B(t) \leq |B|^{\frac{1}{2}} \| u(t) \|_{L^\infty} \leq K \| u_0 \|_{L^\infty} \big( 1 + t \| \omega_0 \|_{L^\infty} \big) = O(t).
\end{equation*}
Note that this inequality has been derived in the case of the Navier-Stokes equations, but the constant $K$ is independent of the viscosity, so it also applies to the Euler equations. It is still open whether this linear bound $O(t)$ is optimal (in Euler or Navier-Stokes), although in the case of Navier-Stokes, the linear bound $O(t)$ has been improved by Gallay and Slijepčević to $O(t^{1/6})$ \cite{GS2014} and to a uniform bound $O(1)$ \cite{GS2015} with constants depending on the viscosity, and provided the fluid evolves in an infinite strip $\Omega = \R \times ]0, L[$.

\subsection{Main Result and Overview of the Paper}

Our main result concerns the global existence and uniqueness of solutions with unbounded initial data of growth $O(|x|^{\frac{1}{2} - \epsilon})$ with bounded vorticity $\omega_0 \in L^\infty$, and without any decay assumption on $\omega_0$. To the best of our knowledge, all previous results of the kind rely on decay assumptions for the vorticity, or at most logarithmic growth of the velocity, or were not global in time (see the references above).

\begin{thm}\label{t:main}
    Assume that the space dimension is $d = 2$ and that $0 \leq \alpha < \frac{1}{2}$. Consider a divergence-free initial datum $u_0 : \R^2 \tend \R^2$ that satisfies the following Morrey-Yudovich type condition (see \eqref{ieq:YudovichSpace} above)
    \begin{equation*}
        u_0 \in L^2_\alpha \qquad \text{and} \qquad \omega_0 = \curl(u_0) \in L^\infty(\R^2).
    \end{equation*}
    In particular, if $u_0$ is such that $\omega_0 \in L^\infty$ and $u_0(x) = O(|x|^\alpha)$ as $|x| \rightarrow \infty$, these conditions are fulfilled. Then the Euler equations have a unique\footnote{It should be understood that the pressure $\pi$ is unique up to the addition of a function of time only.} (weak) solution $(u, \pi)$ associated to that initial datum such that
    \begin{equation}\label{ieq:ThSolutionClass}
        \big\| u(t) \big\|_{L^2_\alpha} \in L^\infty_{\rm loc} (t \in \R_+) \qquad \text{and} \qquad \omega = \curl(u) \in L^\infty(\R_+ ; L^\infty(\R^2)),
    \end{equation}
    and with the far-field condition
    \begin{equation*}
        u(t) - u(0) \in \mc S'_h.
    \end{equation*}
    In addition, define $L^2_{\alpha, 0}$ to be the closure of the space $\mc D$ of smooth and compactly supported functions for the norm $\| \, . \,\|_{L^2_\alpha}$, and define the sets $Y_B$ by
    \begin{equation*}
        Y_B := \Big\{ u_0 \in L^2_{\alpha}, \quad \| \omega_0 \|_{L^\infty} \leq B \Big\}, \qquad B > 0.
    \end{equation*}
    Then, for every $t > 0$, every $B > 0$ and every $u_0 \in L^2_{\alpha, 0} \cap Y_B$, the initial data to solution map $v_0 \longmapsto v(t)$ is continuous at $u_0$ in the metric space $(Y_B, \| \, . \, \|_{L^2_\alpha})$ and $u(t) \in L^2_{\alpha, 0}$. In other words,
    \begin{equation*}
        \forall u_0 \in Y_B \cap L^2_{\alpha, 0}, \forall \epsilon > 0, \exists \delta > 0, \forall v_0 \in Y_B, \qquad \| u_0 - v_0 \|_{L^2_\alpha} \leq \delta \Rightarrow \big\| u(t) - v(t) \big\|_{L^2_\alpha} \leq \epsilon.
    \end{equation*}
    Finally, we have explicit Hölder continuity in weaker $\| \, . \, \|_{L^2_\gamma}$ norms $\gamma > \alpha$ (see Theorem \ref{t:UniquenessYudovich} for precise inequalities).
\end{thm}

\begin{rmk}\label{r:VanishPerturbation}
    Concerning the continuity of the initial data to solution map, the statement above implies that the initial data to solution map is continuous for the metric
    \begin{equation*}
        \big( Y_B \cap L^2_{\alpha, 0}, \| \, . \, \|_{L^2_\alpha} \big) \longrightarrow \big( Y_B \cap L^2_{\alpha, 0}, \| \, . \, \|_{L^2_\alpha} \big).
    \end{equation*}
    However, this is weaker than the assertion contained in Theorem \ref{t:main}, as the initial data to solution map is in fact continuous at every $u_0 \in L^2_{\alpha, 0}$ under $L^2_\alpha$ perturbations with the bounded vorticity constraint $Y_B$ (and not only $L^2_{\alpha, 0}$ perturbations in $Y_B$).
\end{rmk}

Let us make a few comments on Theorem \ref{t:main}. First of all, the restrictions on the exponent $\alpha \in [0, \frac{1}{2}[$ should be understood in relation to the quadratic nature of the pressure when expressed as a function of the velocity: the product $u \otimes u$ has at most linear growth, whereas the pressure force (which we formally express by means of Leray projection)
\begin{equation*}
    \nabla \pi = \sum_{j, k} \nabla (- \Delta)^{-1} \partial_j \partial_k (u_j u_k)
\end{equation*}
involves an operator $\nabla^3 (- \Delta)^{-1}$ of order $1$, which somehow should cancel the linear growth of $u \otimes u$. In our proof this translates as the fact that certain key integrals are convergent, whereas they would not be if $\alpha \geq \frac{1}{2}$. However, the recent construction by Chae and Wolf \cite{CW} of local solutions which possess linear growth (in a Campanato type space $\mc L^1_{1, (1)}$) indicate that square root growth $O(|x|^{1/2 - \epsilon})$ is probably sub-optimal. We will examine this question in a further paper. 

\medskip

The proof of Theorem \ref{t:main} is divided into three main steps, which we shortly describe in this paragraph: \textsl{a priori} estimates, uniqueness of Yudovich regularity solutions, and determination of the correct far-field conditions.

\medskip

\textbf{Existence.} Contrary to many PDE problems, the question of existence of global solutions is not a matter of local regularity. Indeed, the fact that the 2D vorticity $\omega = \partial_1 u_2 - \partial_2 u_1$ solves a pure transport equation
\begin{equation}\label{ieq:OmegaTransport}
    \partial_t \omega + u \cdot \nabla \omega = 0
\end{equation}
immediately gives the bound $\| \omega(t) \|_{L^\infty} = \| \omega_0 \|_{L^\infty}$, so that the solution naturally fulfills global log-Lipschitz estimates. Instead, the main challenge is to prove that there is no blow-up phenomenon created by the growth of the solutions for $|x|\rightarrow \infty$. As an illustration, define the divergence-free functions
\begin{equation*}
    u_n(x) = n^{1 - \alpha} \nabla^\perp \big[ |x|^{1 + \alpha} (1 - \chi(x/n)) \big],
\end{equation*}
where $\chi \in \mc D$ is a smooth compactly supported cut-off function with $\chi(x) = 1$ for $|x| \leq 1$. These functions have uniformly bounded vorticity $\| \omega_n \|_{L^\infty} \leq C$ with respect to $n \geq 1$, but their size nevertheless depends heavily on $n$, even in the space $L^2_\alpha$, as $u_n(x) \approx n |x|^{\alpha - 1} x^\perp$ far away from the origin $|x| \gg n$. In principle, this type of behavior could occur in an unbounded solution of the Euler equations, as kinetic energy is imported from the large scales $|x| \gg 1$ to smaller scales $x \in B(0, 1)$ without there being a change in $\| \omega \|_{L^\infty}$. Our task will be to show that while this norm growth phenomenon cannot be ruled out, it cannot happen too fast either. Our argument will be centered on local energy estimates, where we will show that the flow of energy entering a ball $B(0, R)$ of large radius $R \gg 1$ is negligible when compared to the kinetic energy contained within. In other words, the distribution of kinetic energy on large scales has a slow dependence with respect to time. A consequence is that the flow of energy from large scales to small scales cannot happen too fast, a fact which translates into polynomial time growth for the Morrey norm of the solutions:
\begin{equation*}
    \| u(t) \|_{L^2_\alpha} = O \left( t^{2 \frac{1 + \alpha}{1 - \alpha}} \right) \qquad \text{ as } t \rightarrow \infty.
\end{equation*}
We refer to Corollary \ref{c:algebraic} below for a precise statement. This estimate, which is global with respect to time, yields the existence of global solutions. Our choice of Morrey spaces $L^2_\alpha$ as a functional framework for our result is of course linked to this method: because we will perform energy balance estimates on balls of large radii, the norm \eqref{ieq:YudovichSpace} is particularly well suited for our arguments.

\medskip

\textbf{Uniqueness.} The proof of uniqueness of solutions of the form \eqref{ieq:ThSolutionClass} is an emulation of Yudovich's celebrated result \cite{Yudovich, YudovichEN}, with heavy adaptations to suit the framework of unbounded solutions. The original argument of Yudovich relies on energy estimates and the fact that in 2D the vorticity $\omega$ solves a pure transport equation \eqref{ieq:OmegaTransport}, and so preserves Lebesgue norms $\| \omega (t) \|_{L^p} = \| \omega_0 \|_{L^p}$ (Yudovich assumed that $\omega_0 \in L^1 \cap L^\infty$) as well as the fact that the operator $\omega \mapsto \nabla u = \nabla \nabla^\perp \Delta^{-1}$ satisfies the $L^p$ bound
\begin{equation*}
    \forall p \in ]1, \infty[, \forall \omega \in L^p, \qquad \| \nabla u \|_{L^p} \leq C \frac{p^2}{p-1} \| \omega \|_{L^p}.
\end{equation*}
In the setting of unbounded solutions without decay assumptions on the vorticity, neither the vorticity nor $\nabla u$ have finite $L^p$ norms for $p < \infty$, although the functions are locally $L^p$ ($\nabla u$ is in fact ${\rm BMO}$), not to mention that $u$ cannot be an element of $L^2$ in general. To solve this problem, we develop a local version of Yudovich's argument. However, dealing with unbounded solutions and a possibly non-Lipschitz velocity field creates a new difficulty: because $u$ is non-Lipschitz, perturbations of an initial datum may propagate very fast within the unbounded solution, and we are unable to show that the initial data to solution map is continuous in $L^2_\alpha \longrightarrow L^2_\alpha$. Instead, we must be content with a weaker result, and we prove that, under a bounded vorticity constraint, the initial data to solution map is continuous in $L^2_{\alpha, 0} \longrightarrow L^2_{\alpha, 0}$, where $L^2_{\alpha, 0}$ is the closure in $L^2_\alpha$ of the space of smooth and compactly supported functions (but see Remark \ref{r:VanishPerturbation} for a more precise statement). Our computations rely on estimates in Morrey spaces $L^2_{\beta_t}$ where the growth rate $\beta_t$ depends on time, and on interpolation inequalities, which provide uniform continuity in the weaker $L^2_\gamma \longrightarrow L^2_\gamma$ topology for any $\alpha < \gamma < 1 - \alpha$. Precise bounds can be found in Theorem \ref{t:UniquenessYudovich}.

\medskip

\textbf{Far-field conditions.} As we have explained in Subsection \ref{ss:PressureIntro} above, the pressure is only determined as a solution of the elliptic equation \eqref{ieq:PressurePoisson} up to the addition of a harmonic function, and this creates a non-uniqueness issue, even for regular solutions. This means that any well-posedness result will have to include a selection criterion for solutions of the Poisson equation \eqref{ieq:PressurePoisson} in order to express the pressure force as a function of the velocity $\nabla \pi = \nabla \Pi (u \otimes u)$. In Subsection \ref{ss:PressureIntro}, we have reviewed such criteria from the literature which involve the properties of the pressure $\pi$ or of the Euler solution $u$ (such as far-field behavior).

In Section \ref{s:PressureFormula}, we study this question, and show in Theorem \ref{t:RepresentationPressure} that the particular solution $\nabla \Pi(u \otimes u)$ of the Poisson equation \eqref{ieq:PressurePoisson} we have selected and worked with in all the other paragraphs is equal to the pressure force $\nabla \pi$ if and only if the Euler equation $(u, \pi)$ fulfills one of the two equivalent conditions:
\begin{enumerate}[(i)]
    \item we have $u(t) - u(0) \in \mc S'_h$ for all $t \geq 0$,
    \item we have $\nabla \pi (t) \in \mc S'_h$ for all $t \geq 0$.
\end{enumerate}
These conditions should be understood as far-field conditions, which replace boundary conditions in the context of a PDE defined on the unbounded domain $\R^2$. The fact that they are equivalent to our selection criterion for the solution of \eqref{ieq:PressurePoisson} means that this result is optimal. The proof, which is in the spirit of \cite{Cobb1} (although much more technical due to the unbounded nature of the solutions), relies on kernel estimates for the operator $\nabla^3(- \Delta)^{-1}$ and integral splitting techniques.

\subsubsection*{Layout of the Paper}

To conclude this introduction, let us explain how the paper is organized. After defining some notation and conventions, we present in Section \ref{s:FunctionSpaces} some of the tools we will use throughout the article: the Morrey spaces $L^p_\alpha$, their properties, and various harmonic analysis notions. In Section \ref{s:Definitions}, we define the notion of (unbounded) Yudovich solution and see how it relates to the notion of weak solution of the Euler equations. Section \ref{s:Existence} is occupied with proving \textsl{a priori} estimates for the solution in Morrey spaces and giving a full proof of existence. In Section \ref{s:Uniqueness}, we are preoccupied with proving uniqueness of solutions and studying the continuity properties of the initial data to solution map. Finally, in Section \ref{s:PressureFormula}, we concern ourselves with the matter of uniquely determining the pressure as a solution of a Poisson equation.

\section*{Notation and Conventions}

We gather some of the notation that we will use throughout the paper.

\begin{itemize}
\item In the whole paper, $\theta$, $\chi$ and $\eta$ will be cut-off functions: radial $C^\infty$ functions such that
\begin{equation}\label{eq:cutOff}
0 \leq \chi \leq 1, \qquad \chi(x) = 1 \text{ for } |x| \leq 1, \qquad \text{and} \qquad {\rm supp}(\chi) \subset B(0, 2).
\end{equation}

\item For any $d \geq 1$ and any vector $x \in \R^d$, we denote $\langle x \rangle = (1 + |x|^2)^{1/2}$ the Japanese bracket.

\item If $s \in \R$ is a real number, we denote $s_+ = \max\{ 0, s\}$.

\item Unless otherwise specified, all function spaces should be understood as referring to the space variable only. For example, $L^2 = L^2(\R^d)$. For a time $T > 0$ and a (Fréchet) space $X$, we will use the shorthand $L^\infty_T(X) = L^\infty ([0, T[ ; X)$.

\item We use the Schwartz notation for spaces of test functions and distributions: $\mc D = C^\infty_c$ is the set of smooth and compactly supported functions (on $\R^d$), $\mc D'$ is the set of distributions, $\mc S$ is the class of Schwartz functions and $\mc S'$ is the set of tempered distributions.

\item We denote $D = -i\nabla$. More generally, if $\phi \in \mc S$, then we note $\phi(D)$ the Fourier multiplier operator of symbol $\phi(\xi)$ defined by the relation $\what{\phi(D) f}(\xi) = \phi(\xi) \what{f}(\xi)$.
\end{itemize}

\section*{Acknowledgements}

This work has been supported by Deutsche Forschungsgemeinschaft (DFG, German Research Foundation) Project ID 211504053 - SFB 1060.

\section{Functions Spaces and Harmonic Analysis Lemmas}\label{s:FunctionSpaces}

We start by defining the function spaces that we will use throughout the paper. As we have explained above, our purpose will be to understand solutions of the Euler equations that grow like $O(|x|^\alpha)$ at infinity. While weighted Lebesgue spaces may seem to be a natural tool to handle these, it is more convenient to control the growth of $L^p$ norm on balls of increasing radii, as we intend to perform energy balance estimates in these balls.

\begin{defi}[Morrey spaces]\label{d:MorreySpace}
Consider $\alpha \geq 0$ and $p \in [1, +\infty]$. In accordance with \eqref{ieq:YudovichSpace}, we define a norm by setting, for $f \in L^p_{\rm loc}(\R^d)$,
$$
\| f \|_{L^p_\alpha} := \sup_{r \geq 1} \left[ \frac{1}{r^{d + \alpha p}} \int_{B_r} |f|^p \right]^{1/p} = \sup_{r \geq 1} \left[ \frac{1}{r^{\alpha + d/p}} \| f \|_{L^p(B_r)} \right],
$$
where $B_r = B(0, r)$ denotes the ball of $\R^d$ centered at $x = 0$ and of radius $r$. The function space $L^p_\alpha$ is defined accordingly. Note that the inclusion $L^p_\alpha \subset L^q_\alpha$ holds whenever $p \geq q$. In particular, if $f \in L^\infty_{\rm loc}$ is such that $|f(x)| \lesssim 1 + |x|^\alpha$, then $f \in L^\infty_\alpha \subset L^2_\alpha \subset L^1_\alpha$.

We also define the vanishing Morrey space $L^2_{\alpha, 0}$ to be the closure of the space $\mc D$ of smooth and compactly supported functions in $L^2_\alpha$.
\end{defi}

\begin{rmk}\label{r:Morrey}
    The spaces $L^p_\alpha$ are related to Morrey spaces, although they have their own particularities. For example, all the balls involved in the $L^p_\alpha$ norm are centered at the origin, so that $L^p_\alpha$ is a \textsl{local} Morrey space (we refer to \cite{RSS} for the terminology and a useful overview of Morrey spaces). Furthermore, the radii $r$ of the balls are taken large enough $r \geq 1$, as the $L^p_\alpha$ space is designed to measure the growth of a function at $|x| \rightarrow \infty$ instead of its behavior on small scales. 
\end{rmk}

\begin{rmk}
    The terminology \textsl{vanishing} Morrey space for $L^2_{\alpha, 0}$ comes from the fact that the it can equivalently be defined as the set of all functions $f \in L^2_\alpha$ such that
    \begin{equation*}
        \frac{1}{R^{d + 2 \alpha}} \int_{B_R} |f|^2 \longrightarrow 0 \qquad \text{as } R \rightarrow \infty.
    \end{equation*}
    For any $\beta < \alpha$, we have the inclusion $L^2_\beta \subset L^2_{\alpha, 0}$, thus $L^2_{\alpha, 0}$ can also be seen as the $L^2_\alpha$ closure of the union $\bigcup_{\beta < \alpha} L^2_\beta$.
\end{rmk}

As we have explained in the introduction, the global nature of our result relies on the understanding the distribution of kinetic energy at large scales. For this reason, it is convenient to introduce the following notation: an equivalent norm for Morrey spaces which focuses on large scales. 

\begin{defi}\label{d:EQMorreyNorm}
    If $p \in [1, \infty[$ and $f \in L^p_{\rm loc}$, we define, for a given radius $R \geq 1$, the quantity
    $$
        \| f \|_{L^p_\alpha(R)} := \sup_{r \geq R} \left[ \frac{1}{r^{2 + p \alpha}} \int_{B_r} |f|^p \dx \right]^{1/p}.
    $$
    An obvious adaptation of the above allows to define $\| f \|_{L^\infty_\alpha(R)}$. Of course, the $L^p_\alpha(R)$ norm is equivalent to the $L^p_\alpha$ one, but the equivalence constants depend on $R$. Note that $R \longmapsto \| f \|_{L^p_\alpha(R)}$ is a non-increasing function for any $p \in [1, \infty]$.
\end{defi}

\begin{rmk}
    The $L^2_\alpha(R)$ norm should be understood in relation with the quantity $Z_R(t)$ used in (for instance) \cite{Z, G, CZ}.
\end{rmk}

The Morrey type spaces $L^p_\alpha$ fulfill the following interpolation inequalities, which we will use, and whose proof is straightforward.

\begin{lemma}\label{l:interpolationMorrey}
    Consider $p \in [1, \infty]$ and $\alpha < \gamma < \beta$. Let $\theta \in ]0, 1[$ be such that $\gamma = \theta \beta + (1 - \theta \alpha)$. Then, for any $f \in L^p_\alpha$, we have
    \begin{equation*}
        \| f \|_{L^p_\gamma} \leq \| f \|_{L^p_\beta}^\theta \| f \|_{L^p_\alpha}^{1 - \theta}.
    \end{equation*}
\end{lemma}

We now introduce the class of functions in which we will establish existence and uniqueness of solutions. This will be a space of Yudovich type regularity: the vorticity of the solution is assumed to be bounded while $O(|x|^\alpha)$ growth is still allowed.

\begin{defi}\label{d:YudovichSpace}
Consider $\alpha \in [0, 1[$. We define $Y_\alpha$ to be the space of vector fields $u \in L^2_{\alpha}(\R^d ; \R^d)$ that are divergence-free $\D(u) = 0$ such that
$$
\| u \|_{Y_\alpha} := \| \omega \|_{L^\infty} + \| u \|_{L^2_\alpha} < \infty,
$$
where $\omega = \partial_1 u_2 - \partial_2 u_1$ is the scalar vorticity of $u$.
\end{defi}

In addition, a far-field condition is also needed to insure uniqueness of the solution. This is provided by the following definition (see \cite{Chemin-cours}), which bears on the low frequency properties of functions, and which was formulated by J.-Y. Chemin in the mid 1990s to serve as a basis for realizations of homogeneous Besov spaces (also see \cite{Bourdaud} on that topic).

\begin{defi}\label{d:Chemin}
Fix once and for all a cut-off function $\chi \in C^\infty$ such that
$$
0 \leq \chi \leq 1, \qquad \chi(x) = 1 \text{ for } |x| \leq 1, \qquad \text{and} \qquad {\rm supp}(\chi) \subset B(0, 2).
$$
We then define Chemin's space $\mc S'_h$ of homogeneous distributions by
$$
\mc S'_h := \Big\{ f \in \mc S', \quad \chi(\lambda D) f \tend_{\lambda \rightarrow \infty} 0 \text{ in } \mc S' \Big\}.
$$
\end{defi}

\begin{rmk}\label{r:Chemin}
    There seems to be very little agreement in the mathematical literature on how the space $\mc S'_h$ should be defined. This means that the same notation $\mc S'_h$ has been used for what are in reality different spaces. For example, instead of defining $\mc S'_h$ by a convergence in $\mc S'$ topology, the authors of \cite{BCD} use a convergence in the strong $L^\infty$ topology. We will see that, for our purposes, a weak convergence (in $\mc S'$ for example) is absolutely indispensable. In doing that, we follow \cite{Chemin-cours} and \cite{Danchin-cours}. We refer to \cite{Cobb2} for further discussion on how these different ``versions'' of $\mc S'_h$ compare to each other.
\end{rmk}

\medskip

We recall a standard result on degree zero Fourier multipliers that we will use in the proof. We refer for example to the book \cite{Stein1970} for a (much more general) presentation of this theorem.

\begin{thm}\label{t:czo}
Consider $m \in C^\infty(\R^d \setminus \{0\})$ to be a bounded homogeneous function of degree zero. Then, for any $1 < p < \infty$, the operator $m(D)$ is $L^p \tend L^p$ bounded and the inequality
$$
\| m(D)f \|_{L^p} \leq \frac{p^2}{p - 1} \, C(d) \| f \|_{L^p}
$$
holds for all $f \in L^p$.
\end{thm}

 Finally, we elaborate on a comment that we made earlier: although we have the embedding $L^\infty_\alpha \subset L^2_\alpha$, the reverse is not true, even under smoothness assumptions: in other words $L^2_\alpha \cap \dot{W}^{1, \infty} \not \subset L^\infty_{\alpha}$. We refer to Remark \ref{r:optimalInequality} for an example of this. That being said, the embedding $Y_\alpha \subset L^\infty_{\frac{1 + \alpha}{2}}$ does hold, and it will be crucial to our argument.

\begin{lemma}\label{l:interpolationInequality}
    Consider $0 \leq \alpha \leq 1$ and a function $u \in Y_\alpha$. Then we have $u \in L^\infty_{\frac{1 + \alpha}{2}}$. More precisely, the inequality (see Definition \ref{d:EQMorreyNorm})
    \begin{equation}\label{eq:LinftyAlphaPlus}
        \| u \|_{L^\infty_{\frac{1 + \alpha}{2}}(R)} \lesssim \| \omega \|_{L^\infty}^{1/2} \| u \|_{L^2_\alpha(R)}^{1/2} + \frac{1}{R^{\frac{1 + \alpha}{2}}} \| u \|_{L^2_\alpha(R)}
    \end{equation}
    holds for all $R \geq 1$.
\end{lemma}

\begin{rmk}\label{r:optimalInequality}
    Lemma \ref{l:interpolationInequality} is optimal in the following sense: for any $\alpha < 1$, there is a divergence-free function $u \in Y_\alpha$ such that $u \in L^\infty_\beta$ if and only if $\beta \geq \frac{1 + \alpha}{2}$. We give a construction based on a scaling argument. Consider a non-zero smooth function $\varphi$ supported in the annulus $\{1 \leq |x| \leq 2 \}$, and a $0 \leq \gamma < 0$ whose value will be fixed later, and define the stream function
    \begin{equation*}
        \phi(x) = \sum_{k = 0}^\infty 2^{2k \gamma} \varphi \left( \frac{x - 2^k e}{2^{k \gamma}} \right) := \sum_{k=0}^\infty \phi_k(x),
    \end{equation*}
    where $e \in \R^2$ is a fixed vector of unit norm $|e| = 1$. We set $u = \nabla^\perp \phi$ and $\omega = \curl(u)$. Because we have assumed $\gamma < 1$, each term $\phi_k$ in the series above is supported in 
    \begin{equation*}
        {\rm supp}(\phi_k) \subset \big\{ 2^{k \gamma} \leq |x - 2^k| \leq 2^{k \gamma + 1} \big\} \subset \big\{ 2^{k-N} \leq |x| \leq 2^{k+N} \big\},
    \end{equation*}
    for some $N \geq 1$ whose precise value depends on $\gamma$. As a consequence, for any given $k_0$, there only is a fixed number of the $\phi_k$ whose support may intersect ${\rm supp}(\phi_{k_0})$. As a consequence, we see by direct computation that $u \in L^\infty_\gamma$ (but $u \notin L^\infty_\beta$ if $\beta < \gamma$) and that $\omega \in L^\infty$. On the other hand, for any fixed $n \geq 0$, we have
    \begin{equation*}
        \int_{B_{2^n}} |u|^2 \leq \sum_{k = 0}^{n + N} 2^{2 k \gamma} \int \left| \varphi \left( \frac{x - 2^k e}{2^{k \gamma}} \right) \right|^2 \dx \lesssim 2^{4n\gamma} = (2^n){(2 + 2(2 \gamma - 1)}.
    \end{equation*}
    This shows that $u \in L^2_\alpha$ when $\gamma = \frac{1 + \alpha}{2}$, and so we have an example of a function $u \in Y_\alpha$ with $u \in L^\infty_\beta$ if and only if $\beta \geq \gamma = \frac{1 + \alpha}{2}$.
\end{rmk}

\begin{proof}[Proof (of Lemma \ref{l:interpolationInequality})]
    Consider a cut-off function $\chi \in \mc D (\R^2)$ as in \eqref{eq:cutOff} and define $\chi_R(x) = \chi(x/R)$. We define the function $f = \chi_R u$, whose divergence and curl are
    \begin{equation}\label{eq:DivCurlValue}
        \D(f) = \nabla \chi_R \cdot u \qquad \text{and} \qquad \curl(f) = \nabla^\perp \chi_R \cdot u + \chi_R \omega,
    \end{equation}
    and notice that it is possible to reconstruct $f$ from the divergence and curl through the formula
    \begin{equation*}
        f = \Delta^{-1} \Big( \nabla^\perp \curl(f) + \nabla \D(f) \Big),
    \end{equation*}
    which can be proven by projecting the Fourier transform $\what{f}(\xi)$ on the orthogonal basis $(\xi |\xi|^{-1}, \xi^\perp |\xi|^{-1})$ of the space $\R^2$. We then invoke the (homogeneous) Littlewood-Paley decomposition $f = \sum_{j \in \Z} \dot{\Delta}_j f$ (see Chapter 2 in \cite{BCD}) to write
    \begin{equation*}
        \| f \|_{L^\infty} \leq \sum_{j < N} \| \dot{\Delta}_j f \|_{L^\infty} + \sum_{j \geq N} \| \dot{\Delta}_j f \|_{L^\infty},
    \end{equation*}
    where $N \in \Z$ is a parameter that will be fixed later on. The Bernstein inequalities and standard arguments applied to the Fourier multiplier $\nabla (- \Delta)^{-1}$ (see Lemmas 2.1 and 2.2 in \cite{BCD}) lead to the inequalities
    \begin{equation*}
        \begin{split}
            \| f \|_{L^\infty} & \lesssim \sum_{j < N} 2^{j} \| \dot{\Delta}_j f \|_{L^2} + \sum_{j \geq N} 2^{-j} \Big( \| \dot{\Delta}_j \curl (f) \|_{L^\infty} + \| \dot{\Delta}_j \D (f) \|_{L^\infty} \Big) \\
            & \lesssim 2^N \| f \|_{L^2} + 2^{-N} \Big( \| \curl(f) \|_{L^\infty} + \| \D (f) \|_{L^\infty} \Big).
        \end{split}
    \end{equation*}
    By replacing $f = \chi_R u$ by its value and using \eqref{eq:DivCurlValue}, we deduce from the above that
    \begin{equation*}
        \| u \|_{L^\infty(B_R)} \lesssim 2^N \| u \|_{L^2(B_{2R})} + 2^{-N} \Big( \| \omega \|_{L^\infty} + \| u \|_{L^\infty(B_{2R})} \Big).
    \end{equation*}
    We now pick the value of $N$ such that $2^N \approx \| u \|_{L^2(B_{2R})}^{-\frac{1}{2}} \big( \| \omega \|_{L^\infty} + \| u \|_{L^\infty (B_{2R})} \big)^{\frac{1}{2}}$, so that the inequality above becomes
    \begin{equation*}
        \begin{split}
            \| u \|_{L^\infty(B_R)} & \lesssim \| u \|_{L^2(B_{2R})}^{\frac{1}{2}} \| \omega \|_{L^\infty}^{\frac{1}{2}} + \| u \|_{L^2(B_{2R})}^{\frac{1}{2}} \| u \|_{L^\infty(B_{2R})}^{\frac{1}{2}} \\
            & \lesssim R^{\frac{1 + \alpha}{2}} \| u \|_{L^2_\alpha (R)}^{\frac{1}{2}} \| \omega \|_{L^\infty}^{\frac{1}{2}} + \left( \frac{1}{\epsilon} R^{1 + \alpha} \| u \|_{L^2_\alpha(R)} \right)^{\frac{1}{2}} \left( \epsilon R^{\frac{1 + \alpha}{2}} \| u \|_{L^\infty_{\frac{1 + \alpha}{2}}(R)} \right)^{\frac{1}{2}}.
        \end{split}
    \end{equation*}
    We use Young's inequality $ab \leq a^2 + b^2$ and pick $\epsilon$ small enough to absorb the $L^\infty_{\frac{1 + \alpha}{2}}(R)$ term in the righthand side. By recalling that the function $r \mapsto \| u \|_{L^p_\alpha(r)}$ is non-increasing, we have proven that
    \begin{equation*}
        \forall r \geq R, \qquad \frac{1}{r^{\frac{1 + \alpha}{2}}} \| u \|_{L^\infty(B_r)} \lesssim \| u \|_{L^2_\alpha (R)}^{\frac{1}{2}} \| \omega \|_{L^\infty}^{\frac{1}{2}} + \frac{1}{r^{\frac{1 + \alpha}{2}}} \| u \|_{L^2_\alpha(R)},
    \end{equation*}
    which immediately implies the inequality we were looking for.
\end{proof}

\section{Definition of a Yudovich Solution}\label{s:Definitions}

The purpose of this paragraph is to define a notion of solution in the Yudovich space $Y_\alpha$. The equations of incompressible hydrodynamics are usually solved by eliminating the pressure from the equation and expressing it as a function of the velocity $u$. This is done by means of the Leray projection operator $\P = {\rm Id} + \nabla (- \Delta)^{-1} \D$ so that the Euler equations formally reduce to
$$
\partial_t u + \P \D (u \otimes u) = 0.
$$
The term $\P \D (u \otimes u)$ would then be expressed in terms of a convolution product involving a fundamental solution of the Laplace operator (unique up to a harmonic function)
\begin{equation*}
E(x) = \frac{-1}{2 \pi} \log |x|, \qquad \text{so that } - \Delta E = \delta_0,
\end{equation*}
so that, on a formal level, we would have
\begin{equation*}
    \begin{split}
        \P \D (u \otimes u) & = \D (u \otimes u) + \nabla E * \partial_j \partial_k (u_j u_k) \\
        & = \D (u \otimes u) + \nabla E * (\partial_j u_k \partial_k u_j).
    \end{split}
\end{equation*}
In the above, there is an implicit summation on the repeated indices $j, k = 1, 2$. We will keep this summation convention throughout the rest of the paper. The issue is that the convolution integrals in the formula above do not converge in general if $u \in L^2_\alpha$, even under strong regularity assumptions. Instead of using these convolution products as such, we define an extension of the $\nabla^3(- \Delta)^{-1}$ operator to $Y_\alpha$ that is continuous with respect to the weak-$(*)$ topology.

\begin{defi}\label{d:pressureOperator}
Let $d \geq 2$ and $0 \leq \alpha < \frac{1}{2}$, and consider $\theta \in \mc D$ a cut-off function as in \eqref{eq:cutOff}. For any divergence-free vector fields $u, v \in L^2_\alpha$ such that $\nabla u, \nabla v \in L^p_{\rm loc}$ for some $p > 2$, we define the quantity $\nabla \Pi (u \otimes v)$ by
\begin{equation}\label{eq:PressureYudovichDef}
\nabla \Pi (u \otimes v) := \nabla (\theta E) * (\partial_j u_k \, \partial_k v_j) + \nabla \partial_j \partial_k \big( (1 - \theta) E \big) * (u_j v_k).
\end{equation}
In the above, there is again an implicit sum on the repeated indices $j, k = 1, 2$. The operator $\nabla \Pi(u \otimes v)$ is independent of the choice of the cut-off function $\theta$.
\end{defi}

A few remarks are in order.

\begin{rmk}
    It may not seem obvious here that the convolution integrals involved in $\nabla \Pi(u \otimes v)$ are convergent. However, we will show below in the proofs that it is indeed the case.
\end{rmk}

\begin{rmk}\label{r:pRegularityPressure}
    Of course, if $u$ and $v$ are also sufficiently integrable, say $u, v \in L^p$ for some $p \in ]2, \infty[$, this coincides with the usual expression of the pressure given by Leray projection
    $$
    \nabla \Pi (u \otimes v) = \nabla (- \Delta)^{-1} \partial_j \partial_k (u_j u_k),
    $$
    where the order zero operator $(- \Delta)^{-1} \partial_j \partial_k (u_j u_k)$ can be seen as a Calder\`on-Zygmund operator, more precisely a product of Riesz transforms. Note that the first term on the right hand side of \eqref{eq:PressureYudovichDef} could alternatively be written as $\nabla \partial_j \partial_k \big[ (\theta E) * (u_k u_j) \big]$ which allows to define it for $ u \in L^p_{\rm loc}$ with $p > 2$.
\end{rmk}

\begin{rmk}
    Equation \eqref{eq:PressureYudovichDef} defining the pressure force should be thought of as an instance of pseudo-locality of the operator $\nabla \partial_j \partial_k (- \Delta)^{-1}$. It can be decomposed into the sum of an ``almost local'' operator and an operator with better integrability properties. 
\end{rmk}

Definition \ref{d:pressureOperator} gives a way of expressing the pressure that makes sense for $Y_\alpha$ solutions. We now use this in order to define the notion of Yudovich solution we will be working with.

\begin{defi}[Yudovich solution]\label{d:YudovichSolution}
We fix the space dimension to be $d = 2$. Let $\alpha \in [0, 1/2[$. Consider an initial datum of Yudovich regularity $u_0 \in Y_\alpha$ (see Definition \ref{d:YudovichSpace}). Then a function $u \in L^\infty_{\rm loc} (\R_+ ; Y_\alpha)$ is said to be a Yudovich solution of the Euler equations if it solves the initial value problem
$$
\begin{cases}
\partial_t u + \D (u \otimes u) + \nabla \Pi (u \otimes u) = 0\\
u(0, x) = u_0(x)
\end{cases}
$$
in the weak sense. In other words, for any $\phi \in \mc D (\R_+ \times \R^d ; \R^d)$, we have
\begin{equation*}
\int \phi(0, x) \cdot u_0(x) \dx + \int_0^{ \infty} \int \Big\{ u \cdot \partial_t \phi + \nabla \phi : u \otimes u \Big\} \dx \dt = \int_0^\infty\big\langle \nabla \Pi (u \otimes u), \phi \big\rangle dt.
\end{equation*} 
\end{defi}

It should be noted that one of the particularities of Yudovich solutions as defined above, aside from regularity, is that we have assumed the pressure force to be given by the formula $\nabla \Pi (u \otimes u)$. However, this need not be the case, and we can treat the pressure as an independent unknown of the system. This leads a the notion of solution which is weaker than that above. In fact, we will see that these two notions are not equivalent. Section \ref{s:PressureFormula} below will explore in detail the interplay between these two definitions.

\begin{defi}[Weak solution]\label{d:WeakSolution}
    Let $\alpha \geq 0$ and $T > 0$. Consider $u_0 \in L^2_\alpha$ a divergence-free initial datum. We say that a velocity field $u \in L^2_{\rm loc}([0, T[ ; L^2_\alpha)$ is a weak solution of the Euler equations \eqref{ieq:Euler} if the Euler system with initial datum $u_0$ if it solves the system \eqref{ieq:Euler} in the weak sense. In other words, the following statements hold.
    \begin{enumerate}[(i)]
        \item The divergence equation is solved in the sense of distributions
        \begin{equation*}
            \D(u) = 0 \qquad \text{in } \mc D'(]0, T[ \times \R^2).
        \end{equation*}
        \item The momentum equation is solved in the weak sense: for all test functions $\phi \in \mc D ([0, T[ \times \R^2 ; \R^2)$ such that $\D(\phi) = 0$, we have
        \begin{equation*}
            \int_0^T \int \Big\{ u \cdot \partial_t \phi + u \otimes u : \nabla \phi \Big\} \dx \dt + \int u_0 (x) \cdot \phi (0,x) \dx = 0.
        \end{equation*}
    \end{enumerate}
    In particular, every Yudovich solution is a weak solution.
\end{defi}

\begin{rmk}\label{r:IrregularPressure}
    The conditions of Definition \ref{d:WeakSolution} imply that there exists a distributional pressure force $F_\pi \in \mc D'([0, T[ \times \R^2)$ which is curl free, $\curl(F_\pi) = 0$, and such that
    \begin{equation}\label{eq:WeakEulerDistribution}
        \begin{cases}
            \partial_t u + \D (u \otimes u) + F_\pi = \delta_0 (t) \otimes u_0(x) \\
            \D(u) = 0,
        \end{cases}
        \qquad \text{in } \mc D'([0, T[ \times \R^2),
    \end{equation}
    where the tensor product $\delta_0 (t) \otimes u_0(x)$ is the distribution defined by the following duality bracket:
    \begin{equation*}
        \forall \phi \in \mc D ([0, T[ \times \R^2), \qquad \big\langle \delta_0(t) \otimes u_0(x), \phi(t,x) \big\rangle := \int u_0(x) \phi(0,x) \dx.
    \end{equation*}
    Because $\curl(F_\pi) = 0$, it is tempting to deduce the existence of a pressure function $\pi$ such that $F_\pi = \nabla \pi$ (see for example Lemma 3.11 in \cite{CF3}), however, there is no reason that such a $\pi$ would be a distribution, as it is only defined up to a function of time which can be extremely singular (for example $1/|t| \mathds{1}_{t=0}$ which is not a distribution). To avoid this problem, it is necessary to select a pressure for which this function of time cannot be singular, for example by imposing the condition $\int \pi \chi = 0$, where $\chi$ is a cut-off function as in \eqref{eq:cutOff}. To avoid the hassle of doing this, it is more convenient to work with the pressure force $F_\pi$ than the actual pressure $\pi$ when dealing with weak solutions. 
\end{rmk}

\begin{rmk}\label{r:nonEquivalence}
    In the case of Yudovich solutions, the pressure force $F_\pi$ is given by the operator $\nabla \Pi (u \otimes u)$. However, that is not automatically the case for weak solutions in general: if $u$ is a Yudovich solution and $g : \R_+ \longrightarrow \R^2$ is a non-constant smooth function of time, then we may define a new solution of the Euler equations by non-inertial change of reference frame (see \eqref{ieq:GenGalileo} above). Set
    \begin{equation*}
        v(t,x) = u\big( t, x - G(t) \big) + g(t),
    \end{equation*}
    with $G(t) = \int_0^t g$, then $v$ is a weak solution of the Euler equations according to Definition \ref{d:WeakSolution} for the initial datum $u(0)$, but it is \textsl{not} a Yudovich solution because the associated pressure force is given by $F_\pi = \nabla \Pi(v \otimes v) - g'(t)$, which has zero curl. Consequently, the notion of weak solution from Definition \ref{d:WeakSolution} is weaker than that of Definition \ref{d:YudovichSolution}, even under strong regularity solutions. Also note that this implies that weak solutions of the Euler equations with a given initial datum cannot be unique, since $g$ can be chosen to be zero on a neighborhood of $t=0$. As we will see in Section \ref{s:PressureFormula}, Yudovich solutions satisfy the far-field condition
    \begin{equation*}
        u(t) - u(0) \in \mc S'_h,
    \end{equation*}
    whereas (smooth) solutions according to Definition \ref{d:WeakSolution} do not in general.
\end{rmk}

\begin{rmk}\label{r:timeIrregularity}
    As pointed out in \cite{Cobb1}, weak solutions of the Euler equations suffer from all kind of pathologies when the set of initial data allows for constant flows, as is the case here. For example, if $u \in C^\infty \cap L^\infty ([0, T[ \times \R^d)$ is a smooth and bounded solution of the Euler equations according to Definition \ref{d:WeakSolution} with initial datum $u(0)$, then it also is a solution for the initial data $u(0) + V$ for any $V \in \R^2$. This comes from the fact that the constant velocity field $V$ is both divergence-free and a gradient $V = \nabla (V \cdot x)$, so it orthogonal to any divergence-free test function in the $\mc S' \times \mc S$ duality. This never happens for initial data in $L^p$ with $p < \infty$, as there is no non-trivial divergence-free gradient (\textsl{e.g.} constants) in $L^p$.

    Another odd and unpleasant phenomenon is that solutions according to Definition \ref{d:WeakSolution} have no automatic time continuity properties, even in distributional $\mc D'$ topology. For example, the function of time $u(t) = \mathds{1}_{t \geq T}$ is a solution of the Euler equations (by Definition \ref{d:WeakSolution}) for initial datum $u_0 = 0$, and can even be associated to the pressure $\pi(t,x) = - \delta_T(t) \otimes x$, but is not continuous with respect to time even in $\mc D'$ topology.
\end{rmk}

\begin{rmk}
    The space $\mc D([0, T[ \times \R^2)$ may seem a bit uncomfortable to handle, as the support of its elements is compact in $[0, T[ \times \R^2$, and so may include $t = 0$. A way around this is to extend the solution $u$ to negative times by setting $u(t) = 0$ for $t < 0$. In that case \eqref{eq:WeakEulerDistribution} actually holds in $\mc D' (]- \infty, T[ \times \R^2)$ for some pressure force $F_\pi$. We will use this idea to apply a time regularization by convolution in the proof of Theorem \ref{t:RepresentationPressure}.
\end{rmk}

\section{Existence of Solutions}\label{s:Existence}

In this section, we prove the existence of a global Yudovich solution. As we have explained in the introduction above, the vorticity $\omega$ solves a pure transport equation \eqref{ieq:OmegaTransport}, leading to the bound $\| \omega \|_{L^\infty} = \| \omega_0 \|_{L^\infty}$, and this implies that the solution is naturally globally log-Lipschitz. Therefore, proving global existence of $Y_\alpha$ solutions is not a matter of regularity, but rather showing that the Morrey norm $\| u \|_{L^2_\alpha}$ does not blow up. We start by proving this by studying the local energy balance of solutions, before proving the existence of $Y_\alpha$ solutions.

\subsection{Local Energy Estimates and Global Morrey Bounds}

In this section, we examine local energy estimates for the solution. The goal is to show that the large scale distribution of kinetic energy varies very slowly with respect to time. This large scale information is encoded in the $L^2_\alpha(R)$ norms from Definition \ref{d:EQMorreyNorm}.

\begin{thm}\label{t:EnergyBound}
    Let $0 \leq \alpha < \frac{1}{2}$ and fix the space dimension to $d=2$. Consider an initial datum $u_0 \in Y_\alpha$ and let $u \in L^\infty_{\rm loc} (Y_\alpha)$ be a regular Yudovich solution associated to $u_0$. Then, we have the estimate
    \begin{equation}\label{eq:EnergyUpperBound}
        \| u(t) \|_{L^2_\alpha(R)} \leq 2 \| u_0 \|_{L^2_\alpha},
    \end{equation}
    provided that the radius $R \geq 1$ satisfies the following inequality:
    \begin{equation}\label{eq:RadiusLowerBound}
         R \geq \left( Kt + (Kt)^2 \| \omega_0 \|_{L^\infty} \right)^{\frac{1}{1 - \alpha}} \| u_0 \|_{L^2_\alpha}^{\frac{1}{1 - \alpha}}
    \end{equation}
    where $K > 0$ is some constant that depends only on $\frac{1}{1 - 2 \alpha}$.
\end{thm}


An immediate consequence of Theorem \ref{t:EnergyBound} is a global Morrey bound: the $L^2_\alpha$ norm of the solution is bounded by a polynomial function of time. This will play the role of an \textsl{a priori} estimate.

\begin{cor}\label{c:algebraic}
    With the same assumptions as in Theorem \ref{t:EnergyBound}, we obtain the following global time-growth estimate: for all $t \geq 0$,
    \begin{equation*}
        \| u(t) \|_{L^2_\alpha} \lesssim \| u_0 \|_{L^2_\alpha} + K t^{\frac{1 + \alpha}{1 - \alpha}} \| u_0 \|_{L^2_\alpha}^{\frac{2}{1 - \alpha}} + K t^{2 \frac{1 + \alpha}{1 - \alpha}} \| \omega_0 \|_{L^\infty}^{\frac{1+\alpha}{1 - \alpha}} \| u_0 \|_{L^2_\alpha}^{\frac{2}{1 - \alpha}}.
    \end{equation*}
    The constant $K > 0$ depends, as above, only on $\frac{1}{1 - 2 \alpha}$.
\end{cor}

\begin{rmk}
    For an unbounded solution $u$, it is expected that some growth of $\| u \|_{L^2_\alpha}$ may occur, as can be shown be a simple example. Consider the stationary flow
    $$
    u(x) = \nabla^\perp \left( \langle x \rangle^{1 + \alpha} \right) \in L^\infty_\alpha.
    $$
    Then, for any given $V \in \R^2$, we may consider a Galilean shift of $u$ defined by $v(t, x) = u(x - Vt) + V$, which also solves the Euler equations. As both $v(t) - v(0)$ and $u(t) - u(0)$ are elements of the space $\mc S'_h$, Theorem \ref{t:RepresentationPressure} below insures that it also is a Yudovich solution of the Euler equations. Then we obviously have
    $$
    \| v \|_{L^2(B_1)} \geq 1 + |tV|^\alpha.
    $$
\end{rmk}

\begin{proof}[Proof (of Corollary \ref{c:algebraic})]
    Consider a fixed $R \geq 1$. The proof is mainly a matter of bounding the $L^2_\alpha$ norm by the $L^2_\alpha(R)$ one. Set $f \in L^2_{\rm loc}$. For any $r \geq 1$, we have either, if $r \geq R$,
    $$
    \frac{1}{r^{2 + 2 \alpha}} \int_{B_r} |f|^2 \dx \leq \| f \|_{L^2_\alpha(R)},
    $$
    or, in the case where $r \leq R$,
    $$
    \frac{1}{r^{2 + 2 \alpha}} \int_{B_r} |f|^2 \dx \leq \; \left( \frac{R}{r} \right)^{2 + 2 \alpha} \frac{1}{R^{2 + 2 \alpha}} \int_{B_R} |f|^2 \dx.
    $$
    Consequently, we deduce that for any $R \geq 1$, we must have
    $$
    \| f \|_{L^2_\alpha} \leq R^{1 + \alpha} \| f \|_{L^2_\alpha(R)}.
    $$
    Let us apply this inequality to the results of Theorem \ref{t:EnergyBound}. We take as value for $R$
    $$
    R = 1 + \left( K t \| u_0 \|_{L^2_\alpha} \right)^{\frac{1}{1 - \alpha}} + \left( K t \| \omega_0 \|_{L^\infty}^{1/2} \| u_0 \|_{L^2_\alpha}^{1/2} \right)^{\frac{2}{1 - \alpha}},
    $$
    so that $R$ fulfills inequality \eqref{eq:RadiusLowerBound} while remaining larger than $1$ at all times. By \eqref{eq:EnergyUpperBound}, we may write
    $$
    \| u \|_{L^2_\alpha} \leq R^{1 + \alpha} \| u \|_{L^2_\alpha(R)} \lesssim R^{1 + \alpha} \| u_0 \|_{L^2_\alpha}.
    $$
    Substituting the value of $R$ inside this inequality proves the Corollary.
\end{proof}

\begin{proof}[Proof (of Theorem \ref{t:EnergyBound})]
The main idea of the proof is to perform energy estimates on balls of appropriate radius $R \geq 1$. For this, we fix a smooth cut-off function $\chi \in \mc D$ as in \eqref{eq:cutOff}, which we rescale by setting
\begin{equation}\label{eq:ChiCutOff}
    \chi_R (x) := \chi \left( \frac{x}{R} \right),
\end{equation}
so that $\chi_R$ is supported inside the ball $B_{2R}$. By taking the scalar product of the Euler equation by $\chi_R u$ and integrating over all the space, we obtain a local energy balance
\begin{equation}\label{eq:TotalBalance}
\frac{1}{2} \frac{\rm d}{\dt} \int \chi_R |u|^2 \dx = - \int \chi_R (u \cdot \nabla) u \cdot u \dx - \int \chi_R u \cdot \nabla \Pi (u \otimes u) \dx.
\end{equation}
The integral in the time derivative controls the kinetic energy of the solution in the ball $B_R$. Let us focus on the first integral in the righthand side, which represents roughly the flux of kinetic energy due to convection through the sphere $\partial B_R$. Through integration by parts, we see that
$$
I_1 := - \int \chi_R (u \cdot \nabla) u \cdot u \dx = \int \frac{1}{2} |u|^2 u \cdot \nabla \chi_R \dx
$$
is comparable to flux of kinetic energy through the boundary $\partial B_R$. By writing $\nabla \chi_R (x) = R^{-1} \nabla \chi (x/R)$, which is supported in the ball $B_{2R}$, we see that this integral is bounded by
\begin{equation}\label{eq:EnergyBalanceI1}
|I_1| \leq \frac{1}{2R} \| \nabla \chi \|_{L^\infty} \int_{B_{2R}} |u|^3 \dx \lesssim \frac{R^{2 + 2 \alpha}}{R^{1 - \alpha}} \left( \frac{1}{R^\alpha} \| u \|_{L^\infty(B_{2R})} \right) \left( \frac{1}{R^{1 + \alpha}} \| u \|_{L^2(B_{2R})} \right)^2
\end{equation}
Imagine for a moment that the solution has growth $u(t,x) = O(|x|^\alpha)$, or in other words that $u \in L^\infty_\alpha$ (which is a stronger condition than $u \in L^2_\alpha$). Then, the upper bound above shows that the energy flux $I_1$ due to convection is smaller than the energy in the ball $B_R$ by an order $O \left( \frac{1}{R^{1 - \alpha}} \right)$. In that case, omitting the contribution of the pressure, it would follow that, for any time $t \geq 0$,
\begin{equation}\label{eq:AssymptoticEnergyDistribution}
    \limss_{R \rightarrow \infty} \| u(t) \|_{L^2_\alpha(R)} \leq \| u_0 \|_{L^2_\alpha},
\end{equation}
so that inequality \eqref{eq:EnergyUpperBound} must hold for $R$ large enough (how large will obviously depend on $ \|u(t) \|_{L^\infty_\alpha}$). 

\medskip

However, it is not true in general that $u \in L^\infty_\alpha$, not to mention that obtaining an inequality comparable to \eqref{eq:EnergyBalanceI1} for the contribution of the pressure forces is quite involved. Instead, use the fact that solutions of Yudovich regularity $Y_\alpha$ lie in the space $Y_\alpha \subset L^\infty_{\frac{1 + \alpha}{2}}$, see Lemma \ref{l:interpolationInequality}, which is enough for \eqref{eq:AssymptoticEnergyDistribution} to hold. That being said, we keep estimate \eqref{eq:EnergyBalanceI1} in mind for later and investigate the contribution of the pressure, which involves much more technical computations.

\medskip

\textbf{STEP 1: pressure integral}. We may turn our attention to the pressure integral in \eqref{eq:TotalBalance}, before coming back to \eqref{eq:EnergyBalanceI1}. As we have explained, we must find good estimates for the integrals in the right hand side of \eqref{eq:TotalBalance}. We introduce a new cut-off function $\eta \in \mc D$ such that \eqref{eq:cutOff} holds and we define the functions $\eta_R$ by
\begin{equation}\label{eq:etaR}
    \eta_R(x) := \eta \left( \frac{x}{5R} \right).
\end{equation}
The function $\eta_R$ has support in the ball $B_{10R}$. We use the function $\eta_R$ to separate the tensor product $u \otimes u$ into two parts: one supported in the ball $B_{10R}$, and the other one supported away from $B_{5R}$, namely
$$
\nabla \Pi (u \otimes u) = \nabla \Pi (\eta_R u \otimes u) + \nabla \Pi \big( (1 - \eta_R) u \otimes u \big).
$$
In the equation immediately above, $\nabla \Pi$ is the pressure force map from Definition \ref{d:pressureOperator}.

\medskip

On the one hand, the function $\eta_R u \otimes u$ belongs to $L^2$, so we may take advantage of the fact that $\Pi = (- \Delta)^{-1} \nabla^2$ is a Fourier multiplication operator with a bounded symbol (see Remark \ref{r:pRegularityPressure}). This, along with integration by parts, yields
\begin{equation*}
    \begin{split}
        \left| \int \chi_R u \cdot \nabla \Pi (\eta_R u \otimes u) \dx \right| & = \left| \int \nabla \chi_R \cdot u \, \Pi (\eta_R u \otimes u) \dx \right| \\
        & \leq \frac{1}{R} \| \nabla \chi \|_{L^\infty} \| u \|_{L^2(B_{2R})} \big\| \Pi (\eta_R u \otimes u) \big\|_{L^2(B_{2R})} \\
        & \lesssim \frac{1}{R} \| u \|_{L^2(B_{2R})} \big\| \eta_R u \otimes u \big\|_{L^2}.
    \end{split}
\end{equation*}    
Because the function $\eta_R$ is supported on a ball of radius $10R$, we deduce that
$$
\left| \int \chi_R u \cdot \nabla \Pi (\eta_R u \otimes u) \dx \right| \lesssim \frac{1}{R} \| u \|_{L^2(B_{2R})} \| u \|_{L^\infty(B_{10R})} \| u \|_{L^2(B_{10R})},
$$
and by remembering that $r \longmapsto \| f \|_{L^p_\alpha(r)}$ is a decreasing function, we may write
\begin{equation}\label{eq:EnergyBalanceI2}
    \left| \int \chi_R u \cdot \nabla \Pi (\eta_R u \otimes u) \dx \right| \lesssim \frac{R^{2 + 2 \alpha}}{R^{1 - \alpha}} \| u \|_{L^2_\alpha(2R)}^2 \frac{1}{R^\alpha} \| u \|_{L^\infty (B_{10R})}.
\end{equation}

\medskip

On the other hand, we decompose $\nabla \Pi \big( (1 - \eta_R) u \otimes u \big)$ by using Definition \ref{d:pressureOperator}. Let $\theta \in \mc D$ be a cut-off function as in \eqref{eq:cutOff}, so that ${\rm supp}(\theta) \subset B_{2}$. Then, we have, following the definition of the operator $\nabla \Pi$,
$$
\nabla \Pi \big( (1 - \eta_R) u \otimes u \big) = \nabla (\theta E) * \nabla^2: \big( (1 - \eta_R) u \otimes u \big) + \nabla^3 \big( (1 - \theta) E \big) * \big( (1 - \eta_R) u \otimes u \big).
$$
Recall that this expression of the pressure force $\nabla \Pi (u \otimes u)$ is independent from the choice of cut-off function $\theta$, as we have pointed out in Definition \ref{d:pressureOperator}.

We start by looking at the first term in the righthand side above. This convolution product is supported in the set
\begin{equation*}
    \begin{split}
        {\rm supp} \big[ \nabla (\theta E) * \nabla^2: \big( (1 - \eta_R) u \otimes u \big) \big] & \subset {\rm supp} (\theta E) + {\rm supp} \big( (1 - \eta_R) u \otimes u \big) \\
        & \subset B_2 + {}^c B_{5 R} \subset \big\{ |x| \geq 3R \big\}.
    \end{split}
\end{equation*}
In particular, this function has a support which is disjoint from ${\rm supp}(\chi_R)$. As a consequence, it can have no bearing on the energy estimates and
\begin{equation*}
    \int \chi_R \nabla (\theta E) * \nabla^2 : \big( (1 - \eta_R) u \otimes u \big) \dx = 0.
\end{equation*}
We therefore ignore it and focus on the other term
$$
J(x) := \nabla^3 \big( (1 - \theta) E \big) * \big( (1 - \eta_R) u \otimes u \big)(x).
$$
The convolution kernel $\nabla^3 \big( (1 - \theta)E \big)(y)$ is smooth on the whole space and is therefore is bounded by $C \mathds{1}_{|y| \geq 5R} \langle y \rangle^{-3}$. This provides the following bound for $J(x)$,
\begin{equation*}
    |J(x)| \lesssim \int \big| u \otimes u (y) \big| \frac{\mathds{1}_{|x - y| \geq 5R}}{\langle x - y \rangle^3} \dy.
\end{equation*}
We then decompose this integral in a zone where the lower bound $\langle x-y \rangle \geq \frac{1}{C} \langle y \rangle$ holds and a zone where it does not. More precisely, by writing $\R^2 = \{ |y| \leq 2|x| \} \sqcup \{ |y| > 2|x| \}$, we have
\begin{equation*}
\begin{split}
|J(x)| & \lesssim \int_{|y| \leq 2|x|} \big| u \otimes u (y) \big| \frac{\mathds{1}_{|y| \geq 5 R }}{\langle x-y \rangle^3} \dy + \int_{|y| \geq 2|x|} \big| u \otimes u (y) \big| \frac{\mathds{1}_{|y| \geq 5 R }}{\langle x-y \rangle^3} \dy \\
& \lesssim \int \big| u \otimes u (y) \big| \frac{\mathds{1}_{5R \leq |y| \leq 2|x|}}{\langle x-y \rangle^3} \dy + \int_{|y| \geq 2|x|} \big| u \otimes u (y) \big| \frac{\mathds{1}_{|y| \geq 5 R }}{\langle y \rangle^3} \dy.
\end{split}
\end{equation*}
The first of these two integrals is defined on the set $5R \leq |y| \leq 2|x|$, and must therefore vanish if $|x| \leq 5/2R$. This means that when $J(x)$ is multiplied by $\chi_R$, this integral has no bearing whatsoever on the computations, and we may only consider the integral $|y| \geq 2|x|$. By using the lower bound $\langle x-y \rangle \geq \frac{1}{2}\langle y \rangle$, which holds on $\{ |y| \geq 2|x| \}$ by triangular inequality, we have
\begin{equation*}
    \chi_R(x) |J(x)| \lesssim \int_{|y| \geq 2|x|} \big| u \otimes u(y) \big| \frac{\mathds{1}_{|y| \geq 5R}}{\langle y \rangle^3} \dy,
\end{equation*}
and we are reduced to an integral whose integrand depends only on $y$. Now, estimating this integral in terms of $\| u \|_{L^2(B_{CR})}$ only, as we did for the other integrals above, is impossible, as all values of $u(y)$ for $|y| \geq 5R$ are involved, so we have to involve the full norm $\| u \|_{L^2_\alpha(CR)}$ in our estimates. We will do this by introducing a dyadic partition of the space: set
\begin{equation}\label{eq:dydadicSplittingR3Existence}
\begin{split}
\{ |y| \geq 5R \} \subset \{ |y| \geq R \} & = \bigsqcup_{k=0}^\infty \big\{ 2^k R \leq |y| < 2^{k+1} R \big\} \\
& := \bigsqcup_{k=0}^\infty A_k(R).
\end{split}
\end{equation}
We may then write
\begin{equation}\label{eq:SumBoundsExistence}
    \begin{split}
        \chi_R(x) |J(x)| & \lesssim \int_{|y| \geq 2|x|} \big| u \otimes u (y) \big| \frac{\mathds{1}_{|y| \geq 5 R }}{\langle y \rangle^3} \dy \\
        & \lesssim \sum_{k = 0}^\infty \int_{A_k(R)} \big| u \otimes u (y) \big| \frac{\dy}{\langle y \rangle^3} \\
        & \lesssim \sum_{k=0}^\infty \frac{1}{\left( 2^{k+1} R \right)^3}  \int_{A_k(R)} |u(y)|^2 \dy \\
        & = \sum_{k=0}^\infty \frac{\left( 2^{k+1}R \right)^{2 + 2\alpha}}{\left( 2^{k+1} R \right)^3} \left( \frac{1}{\left( 2^{k+1}R \right)^{2 + 2\alpha}} \int_{B_{2^{k+1}R}} |u(y)|^2 \dy \right)
    \end{split}
\end{equation}
The integral in the parenthesis can be bounded by using the norm $\| u \|_{L^2_\alpha(2R)}^2$, leaving us with a geometric sum to compute, and we obtain
\begin{equation}\label{eq:ThirdIntegralEstimate}
    \chi_R(x) |J(x)| \lesssim \frac{1}{1 - 2 \alpha} \frac{1}{R^{1 - 2 \alpha}} \| u \|_{L^2_\alpha(2R)}^2.
\end{equation}
By plugging this in the pressure integral, we find the upper bound we were after:
\begin{equation}\label{eq:EnergyBalanceI3}
\begin{split}
\left| \int \chi_R u \cdot \nabla \Pi \big( (1 - \eta_R) u \otimes u \big) \dx \right| & \lesssim \frac{1}{R^{1 - 2 \alpha}} \| u \|^2_{L^2_\alpha(2R)} \int \chi_R |u| \dx \\
& \lesssim \frac{R^{2 + 2\alpha}}{R^{1 - \alpha}} \| u \|_{L^2_\alpha(2R)}^2 \| u \|_{L^1_\alpha(2R)}.
\end{split}
\end{equation}

\medskip

\textbf{STEP 2: summary and reduction to energy norms}. Let us summarize the estimates we have obtained so far. By putting together \eqref{eq:EnergyBalanceI1}, \eqref{eq:EnergyBalanceI2} and \eqref{eq:EnergyBalanceI3}, we are able to deduce from the energy balance \eqref{eq:TotalBalance} a differential inequality, namely
\begin{equation}\label{eq:DiffEQ1}
    \frac{\rm d}{\dt} \left\{ \frac{1}{R^{2 + 2 \alpha}} \int \chi_R |u|^2 \dx \right\} \lesssim \frac{1}{R^{1 - \alpha}} \| u \|^2_{L^2_\alpha(2R)} \left( \| u \|_{L^1_\alpha(2R)} + \frac{1}{R^\alpha} \| u \|_{L^\infty(B_{10R})} \right).
\end{equation}
The problem with the previous inequality is that the upper bound we have obtained does not only involve $L^2_\alpha(2R)$ norms. The first summand appearing in the righthand side is a $L^1_\alpha$ type quantity, but this is no problem because we have, due to Hölder's inequality,
$$
\| u \|_{L^1_\alpha(2R)} \lesssim \| u \|_{L^2_\alpha(2R)}.
$$
However, for the $L^\infty$ norm, we must resort to the embedding $Y_\alpha \subset L^\infty_{\frac{1 + \alpha}{2}}$ of Lemma \ref{l:interpolationInequality}, and in particular inequality \eqref{eq:LinftyAlphaPlus} therein. This allows us to rewrite \eqref{eq:DiffEQ1} as (recall that $\| f \|_{L^2_\alpha(2R)} \leq \| f \|_{L^2_\alpha(R)}$)
\begin{equation*}
    \frac{\rm d}{\dt} \left\{ \frac{1}{R^{2 + 2 \alpha}} \int \chi_R |u|^2 \dx \right\} \lesssim \| u \|^2_{L^2_\alpha(R)} \left\{ \frac{1}{R^{1 - \alpha}} \| u \|_{L^2_\alpha(R)} + \frac{\| \omega \|_{L^\infty}^{1/2}}{R^{\frac{1 - \alpha}{2}}} \| u \|^{1/2}_{L^2_\alpha(R)} \right\},
\end{equation*}
and by integrating with respect to time, 
\begin{equation*}
    \begin{split}
        \frac{1}{R^{2 + 2 \alpha}} \int_{B_R} |u(t,x)|^2 \dx & \leq \frac{1}{R^{2 + 2 \alpha}} \int \chi_R(x) |u(t,x)|^2 \dx \\
        & \leq \| u_0 \|_{L^2_\alpha (R)}^2 + C \int_0^t \| u \|^2_{L^2_\alpha(R)} \left( \frac{1}{R^{1 - \alpha}} \| u \|_{L^2_\alpha(R)} + \frac{\| \omega \|_{L^\infty}^{1/2}}{R^{\frac{1 - \alpha}{2}}} \| u \|^{1/2}_{L^2_\alpha(R)} \right) {\rm d} \tau.
    \end{split}
\end{equation*}
Because the function $r \longmapsto \| f \|_{L^2_\alpha(r)}$ is non-increasing, the inequality above shows that, for any fixed $R \geq 1$ and any $r \geq R$, we have
\begin{equation*}
    \frac{1}{r^{2 + 2 \alpha}} \int_{B_r} \big| u(t) \big|^2 \dx \leq \| u_0 \|_{L^2_\alpha(R)}^2 + C \int_0^t \| u \|^2_{L^2_\alpha(R)} \left( \frac{1}{R^{1 - \alpha}} \| u \|_{L^2_\alpha(R)} + \frac{\| \omega \|_{L^\infty}^{1/2}}{R^{\frac{1 - \alpha}{2}}} \| u \|^{1/2}_{L^2_\alpha(R)} \right) {\rm d} \tau
\end{equation*}
so that, by taking the supremum over all $r \geq R$, we finally get an integral inequality for the quantity $\| u \|_{L^2_\alpha(R)}$, namely
\begin{equation}\label{eq:IntIneq1}
    \big\| u(t) \big\|_{L^2_\alpha(R)}^2 \leq \| u_0 \|_{L^2_\alpha(R)}^2 + C \int_0^t \| u \|^2_{L^2_\alpha(R)} \left( \frac{1}{R^{1 - \alpha}} \| u \|_{L^2_\alpha(R)} + \frac{\| \omega \|_{L^\infty}^{1/2}}{R^{\frac{1 - \alpha}{2}}} \| u \|^{1/2}_{L^2_\alpha(R)} \right) {\rm d} \tau
\end{equation}

\begin{rmk}
    In particular, the upper bound in the integral inequality above tends to zero as $R \rightarrow \infty$, and we deduce that we indeed have $\sup_{[0, T]} \limss_{R\to \infty} \| u(t) \|_{L^2_\alpha(R)} \leq \| u_0 \|_{L^2_\alpha}$ for any $T > 0$ such that $u \in L^\infty ([0, T[ ; L^2_\alpha)$, as we had expected from the beginning of the proof. 
\end{rmk}

\medskip

\textbf{STEP 3: closing the estimates}. Now, we just have to use \eqref{eq:IntIneq1} to prove the conclusion. Consider any time $t > 0$. We seek a condition that the radius $R > 0$ must satisfy in order to have $\| u(\tau) \|_{L^2_\alpha} \leq 2 \| u_0 \|_{L^2_\alpha}$ for all $\tau \in [0, t]$. Therefore, define the time $t^*$ by
\begin{equation*}
    t^* := \sup \left\{ T \in [0, t], \quad C \int_0^T \| u \|^2_{L^2_\alpha(R)} \left( \frac{1}{R^{1 - \alpha}} \| u \|_{L^2_\alpha(R)} + \frac{\| \omega \|_{L^\infty}^{1/2}}{R^{\frac{1 - \alpha}{2}}} \| u \|^{1/2}_{L^2_\alpha(R)} \right) {\rm d} \tau \leq \| u_0 \|_{L^2_\alpha}^2 \right\}.
\end{equation*}
The time $t^*$ is defined so that we have $\| u(\tau) \|_{L^2_\alpha} \leq 2 \| u_0 \|_{L^2_\alpha}$ for all $\tau \leq t$. Now, for all times $T \leq t^*$, we must have the inequality
\begin{multline*}
    C \int_0^T \| u \|^2_{L^2_\alpha(R)} \left( \frac{1}{R^{1 - \alpha}} \| u \|_{L^2_\alpha(R)}  + \frac{\| \omega \|_{L^\infty}^{1/2}}{R^{\frac{1 - \alpha}{2}}} \| u \|^{1/2}_{L^2_\alpha(R)} \right) {\rm d} \tau \\
    \leq 8 C \| u_0 \|_{L^2_\alpha}^2 \left( \frac{T}{R^{1 - \alpha}} \| u_0 \|_{L^2_\alpha} + \frac{T}{R^{\frac{1 - \alpha}{2}}} \| \omega_0 \|_{L^\infty}^{1/2} \| u_0 \|_{L^2_\alpha}^{1/2} \right).
\end{multline*}
In particular, we see that if $R \geq 1$ is such that
\begin{equation*}
    R \geq \left( 8 C t \| u_0 \|_{L^2_\alpha} \right)^{\frac{1}{1 - \alpha}} + \left( 8 C t \| \omega_0 \|_{L^\infty}^{1/2} \| u_0 \|_{L^2_\alpha}^{1/2} \right)^{\frac{2}{1 - \alpha}},
\end{equation*}
then we must have $t^* = t$. Indeed, if that were not the case, and $t^* < t$, then our choice of $R$ would show that the inequality in the supremum defining $t^*$ is strict for all $T \in [0, t^*]$, thus contradicting the definition of the supremum. This ends the proof of Theorem \ref{t:EnergyBound}.

\end{proof}

\subsection{Existence of Solutions}

In this subsection, we establish the existence of global $Y_\alpha$ solutions by means of an approximation scheme: by truncating the initial data, we use the classical existence of Yudovich solutions (for $\omega \in L^1 \cap L^\infty$) to construct a family of approximate solutions which will satisfy the bounds of the previous paragraph. Convergence to a $Y_\alpha$ solution will be obtained by means of a compactness method.

\begin{thm}
    Consider $0 \leq \alpha < \frac{1}{2}$ and an initial datum $u_0 \in Y_\alpha$. Then there exists a Yudovich solution $u \in L^\infty_{\rm loc}(\R_+ ; Y_\alpha)$ associated to that initial datum, according to Definition \ref{d:YudovichSolution}.
\end{thm}

\begin{proof}
    Consider a truncation function $\chi \in \mc D$ as in \eqref{eq:cutOff} and define the sequence $(u_n, \pi_n)$ as being the unique Yudovich solutions of the 2D Euler equations (see \cite{Yudovich, YudovichEN} or Chapter 8 of \cite{MB}) with $\omega_{n}(0) \in L^1 \cap L^\infty$ and truncated initial data:
    \begin{equation*}
        \begin{cases}
            \partial_t u_n + \D(u_n \otimes u_n) + \nabla \pi_n = 0 \\
            \D(u_n) = 0,
        \end{cases}
        \qquad \text{with } u_n(0, x) = \chi \left( \frac{x}{n} \right) u_0(x).
    \end{equation*}
    In that case, the pressure force is given by the formula of Definition \ref{d:pressureOperator}, namely $\nabla \pi_n = \nabla \Pi (u_n \otimes u_n)$, so that $(u_n)$ is a Yudovich $L^2_\alpha$ solution of the Euler problem according to Definition \ref{d:YudovichSolution}.

    \medskip

    The global bounds of Corollary \ref{c:algebraic} show that the sequence $(u_n)$ satisfies the following uniform bounds: for any fixed $T > 0$, 
    \begin{equation}\label{eq:UnifBoundsExistence}
        \| \omega_n \|_{L^\infty(\R_+ \times \R^2)} \leq C \qquad \text{and} \qquad \| u_n \|_{L^\infty_T(L^2_\alpha)} \leq C.
    \end{equation}
    These bounds are enough to provide compactness with respect to the space variable. However, for compactness in the time variable, we need to find some uniform bound on the time derivatives $(\partial_t u_n)$. We start with localizing the problem by fixing a cut-off function $\chi_R \in \mc D$ as in \eqref{eq:ChiCutOff} and note that
    \begin{equation*}
        \partial_t (\chi_R u_n) + \chi_R \D(u_n \otimes u_n) + \chi_R \nabla \Pi (u_n \otimes u_n) = 0.
    \end{equation*}
    On the one hand, the uniform bounds \eqref{eq:UnifBoundsExistence} combined with inequality \eqref{eq:LinftyAlphaPlus} allow us to infer that the sequence $(u_n)$ is bounded in the space $L^\infty_T (L^2_\alpha \cap L^\infty_{(1 + \alpha)/2})$. In particular, the tensor product $u_n \otimes u_n$ is bounded in $L^\infty_T(L^2_{\rm loc})$, and we deduce the bound
    \begin{equation*}
        \big\| \chi_R \D (u_n \otimes u_n) \big\|_{L^\infty_T (H^{-1})} \leq C(T, R).
    \end{equation*}
    In the above, the constant depends on $T$ and $R \geq 1$, but that does not matter as long as it does not depend on $n$. Estimates for the pressure force $\nabla \Pi (u_n \otimes u_n)$ are a bit more complicated as the operator $\nabla \Pi$ is non-local, so that we need to shortly revisit the computations of the previous section. Let $\theta \in \mc D$ be a cut-off function as in \eqref{eq:cutOff} and recall $\eta_R$ from \eqref{eq:etaR}. We have, as before,
    \begin{equation*}
        \chi_R \nabla \Pi (u_n \otimes u_n) = \chi_R \nabla \Pi (\eta_R u_n \otimes u_n) + \chi_R \nabla \Pi \big( (1 - \eta_R) u_n \otimes u_n \big).
    \end{equation*}
    The first term can be bounded by using the fact that $\Pi = (- \Delta)^{-1} \nabla^2$ is a Fourier multiplication operator with a bounded symbol. We therefore have
    \begin{equation*}
        \begin{split}
            \big\| \chi_R \nabla \Pi \big( \eta_R u_n \otimes u_n \big) \big\|_{H^{-1}} & \lesssim \| \eta_R u_n \otimes u_n \|_{L^2}\\
            & \lesssim \| \eta_R^{1/2} u_n \|_{L^2} \| \eta_R^{1/2} u_n \|_{L^\infty} \\
            & \lesssim C(R) \| u_n \|_{L^2_\alpha} \| u_n \|_{L^\infty_{\frac{1 + \alpha}{2}}} \leq C(T, R).
        \end{split}
    \end{equation*}
    For the other pressure term, we use inequality \eqref{eq:ThirdIntegralEstimate} in order to obtain the bound
    \begin{equation*}
        \chi_R \big| \nabla \Pi \big( (1 - \eta_R) u_n \otimes u_n \big) \big| \leq \| u_n \|_{L^2_\alpha}^2 \leq C(T,R),
    \end{equation*}
    which gives the uniform bound
    \begin{equation*}
        \big\| \chi_R \nabla \Pi \big( (1 - \eta_R) u_n \otimes u_n \big) \big\|_{L^\infty_T(L^2)} \leq C(T, R).
    \end{equation*}
    In conclusion, we see that the time derivatives $\partial_t (\chi_R u_n)$ are bounded in the space $L^\infty_T(H^{-1})$. In parallel, the uniform bounds of \eqref{eq:UnifBoundsExistence} show that the functions $\chi_R u_n$ have a vorticity
    \begin{equation*}
        \curl (\chi_R u_n) = \chi_R \omega_n + \nabla^\perp \chi_R \cdot u_n
    \end{equation*}
    that is uniformly bounded in the space $L^\infty_T(L^2)$ with respect to $n$. Putting everything together, we obtain the following uniform bounds
    \begin{equation*}
        \| \chi_R u_n \|_{W^{1, \infty}_T(H^{-1})} + \| \chi_R u_n \|_{L^\infty_T(H^1)} \leq C(T, R)
    \end{equation*}
    hold. The compact embedding $H^1_{\rm loc} \subset L^p_{\rm loc}$ for $2 \leq p < \infty$, Ascoli's theorem, and the fact that $R \geq 1$ is arbitrary allow us to deduce convergence of the sequence $(u_n)$, possibly up to an extraction of a subsequence,
    \begin{equation}\label{eq:StrongConvergence}
        u_n \tend_{n \rightarrow \infty} u \qquad \text{in } L^\infty_T(L^p_{\rm loc}) \text{ and a.e.}
    \end{equation}
    
    \medskip

    Now, let us check that the limit $u$  belongs to the Yudovich space $L^\infty_T(Y_\alpha)$. On the one hand, it is fairly straightforward to see that $u \in L^\infty_T(L^2_\alpha)$. Indeed, the strong convergence \eqref{eq:StrongConvergence} in $L^\infty_T(L^2_{{\rm loc}})$ shows that we have
    \begin{equation*}
        \| u_n(t) \|_{L^2(B_R)} \longrightarrow \| u(t) \|_{L^2(B_R)} \leq C(T) R^{1 + \alpha}
    \end{equation*}
    for almost every $0 \leq t \leq T$. In addition, the Banach-Steinhaus theorem insures that we also preserve the vorticity bound $\| \omega \|_{L^\infty} \leq \| \omega_0 \|_{L^\infty}$. We see that the limit $u$ is indeed an element of the Yudovich space $L^\infty_T(Y_\alpha)$.

    \medskip

    Finally, we have to make sure that $u$ is indeed a solution of the equation. Thanks to the convergence \eqref{eq:StrongConvergence}, we immediately get
    \begin{equation}\label{eq:EasyCV}
        \begin{split}
            &u_n(0) \longrightarrow u_0  \qquad \text{in } \mc D'(\R^2)\\
            &\partial_t u_n \longrightarrow \partial_t u \qquad \text{in } \mc D'(]0, T[ \times \R^2) \\
            &\D (u_n \otimes u_n) \longrightarrow \D (u \otimes u) \qquad \text{in } \mc D'(]0, T[ \times \R^2),
        \end{split}
    \end{equation}
    and it remains to check that the pressure $\nabla \Pi (u_n \otimes u_n)$ also converges to $\nabla \Pi (u \otimes u)$. By Definition \ref{d:pressureOperator}, we have 
    \begin{equation*}
        \nabla \Pi (u_n \otimes u_n) = \nabla (\theta E) * \nabla^2 : (u_n \otimes u_n) + \nabla^3 \big( (1 - \theta E) \big) * (u_n \otimes u_n).
    \end{equation*}
    Concerning the first term, the kernel $\nabla (\theta E)$ is compactly supported and lies in $L^q$ for $1 \leq q < 2$. Therefore, by using the convergence \eqref{eq:StrongConvergence} for $p > 4$, we see that
    \begin{equation*}
        \nabla (\theta E) * (u_n \otimes u_n) \longrightarrow \nabla (\theta E) * (u \otimes u) \qquad \text{in } L^1_{\rm loc}(]0, T[ \times \R^2),
    \end{equation*}
    and therefore we have the distributional convergence of the second derivatives
    \begin{equation}\label{eq:CVPressure1}
        \begin{split}
            \nabla (\theta E) * \nabla^2 : (u_n \otimes u_n) & = \nabla^2 : \big( \nabla (\theta E) * (u_n \otimes u_n) \big)\\
            & \longrightarrow \nabla^2 : \big( \nabla (\theta E) * (u \otimes u) \big) \qquad
         \text{in } \mc D'(]0, T[ \times \R^2)\\
            & = \nabla (\theta E) * \nabla^2 : (u \otimes u).
        \end{split}
    \end{equation}
    Concerning the second term in the pressure operator, we note that because of inequality \eqref{eq:ThirdIntegralEstimate}, we have the following uniform bound at our disposal:
    \begin{equation*}
        \chi_R(x) \big| \nabla^3 \big( (1 - \theta)E \big) * (u_n \otimes u_n)(x) \big| \lesssim C(R) \| u_n \|_{L^2_\alpha}^2 \leq C(T, R).
    \end{equation*}
    Therefore, it is enough to show convergence for almost every $(t, x) \in [0, T[ \times \R^2$ to show that the above also converges in $L^1$, thanks to the dominated convergence theorem (note that $\chi_R$ has uniformly compact support). To facilitate this, we decompose again the integral with the help of a dyadic partition of the space. Define
    \begin{equation*}
        \begin{split}
            \R^2 & = \{ |y| \leq 1 \} \sqcup \bigsqcup_{k=0}^\infty \{ 2^k < |y| \leq 2^{k+1} \} \\
            & := A_0 \sqcup \bigsqcup_{k=1}^\infty A_k.
        \end{split}
    \end{equation*}
    Then we have, by noting $\Psi = \nabla^3 ((1 - \theta) E)$ the convolution kernel,
    \begin{equation}\label{eq:SummationLimit}
        \chi_R (x) \big( \Psi * (u_n \otimes u_n)(x) \big) = \sum_{k=0}^\infty \int_{A_k} u_n \otimes u_n (y) \, \Psi (x-y) \dy.
    \end{equation}
    Now, thanks to the strong convergence \eqref{eq:StrongConvergence}, each integral in this series converges for almost every $(t,x)$, possibly up to an extraction of a subsequence:
    \begin{equation}\label{eq:SummandLimit}
        I_{n, k}(t,x) := \int_{A_k} u_n \otimes u_n (t, y) \, \Psi (x-y) \dy \longrightarrow \int_{A_k} u \otimes u (t,y) \, \Psi (x-y) \dy,
    \end{equation}
    while at the same time satisfying the uniform bounds $|I_{n, k}(t,x)| \leq 2^{-k(1 - 2 \alpha)}C(T,R)$ from \eqref{eq:SumBoundsExistence}. A final application of dominated convergence in $\ell^1(k \geq 0)$ shows that we may take the limit \eqref{eq:SummandLimit} under the summation symbol in \eqref{eq:SummationLimit}, and therefore obtain
    \begin{equation}\label{eq:CVPressure2}
        \chi_R (x) \big( \Psi * (u_n \otimes u_n)(t,x) \big) \longrightarrow \chi_R (x) \big( \Psi * (u \otimes u) (t,x) \big) \qquad \text{a.e.}
    \end{equation}
    Once again, since the function $\chi_R$ has compact support, the almost everywhere convergence implies that \eqref{eq:CVPressure2} actually holds in the $L^1([0, T[\times \R^2)$ topology. By adding both convergences \eqref{eq:CVPressure1} and \eqref{eq:CVPressure2}, we obtain the desired property
    \begin{equation*}
        \nabla \Pi (u_n \otimes u_n) \longrightarrow \nabla \Pi (u \otimes u) \qquad \text{in } \mc D' ([0, T[ \times \R^2).
    \end{equation*}
    Together with \eqref{eq:EasyCV}, this shows that $u$ is in fact a Yudovich solution of the Euler equations associated to the initial datum $u_0$, as in Definition \ref{d:YudovichSolution}. The proof of existence is completed.
\end{proof}

\section{Uniqueness of Yudovich Solutions}\label{s:Uniqueness}

In this Section, we study the uniqueness of Yudovich solutions. As we have explained in the introduction, the method of proof is directly inspired by the original proof of Yudovich, who obtained uniqueness for energy solutions $u \in L^\infty(L^2)$ with vorticities lying in $\omega \in L^\infty(L^1 \cap L^\infty)$. However, in the case of unbounded solutions  $u \in L^\infty_{\rm loc}(L^2_\alpha)$ without any decay restriction on the vorticity, these arguments need some heavy adaptation, as none of the integrals involved in the original proof of Yudovich actually exist. 

\medskip

In addition, the log-Lipschitz regularity of the velocity field ($\nabla u$ is in ${\rm BMO}$) creates further problems when seeking stability estimates: because of it perturbations of the initial datum propagate at high speed in the fluid, and this has made it impossible for us to show that the initial data to solution map  $L^2_\alpha \longrightarrow L^2_\alpha$ is continuous. Instead, we prove that it H\"older continuous on bounded sets in $Y_\alpha$ in the topology of  $L^2_\gamma \longrightarrow L^2_\gamma$  for any $\gamma > \alpha$ such that $1 - \alpha - \gamma > 0$. This implies (non uniform)  continuity in the $L^2_{\alpha}$  topology on bounded sets of $Y_\alpha$ intersected with 
$ L^2_{\alpha, 0}$.

\begin{thm}\label{t:UniquenessYudovich}
Consider two Yudovitch solutions $u_1, u_2 \in L^\infty_{\rm loc}(Y_\alpha)$ (see Definition \ref{d:YudovichSolution}) associated to the initial data $u_{0, 1}, u_{0, 2} \in Y_\alpha$ such that the respective pressures are given by the integral representation \eqref{eq:PressureYudovichDef}. We also define 
\begin{equation*}
    K(t) = \sup_{[0, t]} \Big( \| u_1 \|_{Y_\alpha} + \| u_2 \|_{Y_\alpha} \Big)
\end{equation*}
which in Corollary \ref{c:algebraic}  is  bounded by a function of $\Vert u_{0,1} \Vert_{Y_\alpha} + \Vert u_{0,2} \Vert_{Y_\alpha} $and a constant $C(\alpha, t)$. The following statements hold:
\begin{enumerate}[(i)]
    \item For any $\gamma > \alpha$ such that $1 - \alpha - \gamma > 0$ and any time $t > 0$, the initial data to solution map $v_0 \longmapsto v(t)$ is $L^2_\gamma \longrightarrow L^2_\gamma$ uniformly continuous on bounded sets of $Y_\alpha$. More precisely, we have the following local Hölder estimate
    \begin{equation}\label{eq:localHolder}
        \| u_2(t) - u_1(t) \|_{L^2_\gamma} \lesssim \| u_{0, 2} - u_{0, 1} \|_{L^2_\gamma}^{ e^{-tCK(t)}} K(t)^{1 - e^{-tCK(t)}},
    \end{equation}
    where the constants in the above only depend on $\alpha$ and $\gamma$.

    \item For every $B > 0$, define the set $Y_B$ by 
    \begin{equation*}
        Y_B := \Big\{ u_0 \in L^2_\alpha, \quad \| \omega_0 \|_{L^\infty} \leq B \Big\}.
    \end{equation*}
    Then for every $u_0 \in Y_B \cap L^2_{\alpha, 0}$ (see Definition \ref{d:MorreySpace} for $L^2_{\alpha, 0}$) the initial data to solution map $v_0 \longmapsto v(t)$ is continuous at $u_0$ in the metric space $(Y_B, \| \, . \, \|_{L^2_\alpha})$ and we have $u(t) \in L^2_{\alpha, 0}$ (see Remark \ref{r:VanishPerturbation}).
\end{enumerate}
\end{thm}

\begin{proof}[Proof (of Theorem \ref{t:UniquenessYudovich})]
We denote $w = u_2 - u_1$ and $w_0 = u_{0, 2} - u_{0, 1}$. We also consider a cut-off function $\chi \in \mc D$ as in \eqref{eq:cutOff} and define, for any $R \geq 1$,
$$
\chi_R(x) := \chi \left( \frac{x}{R} \right), \qquad \text{so that } {\rm supp}(\chi_R) \subset B_{2R}.
$$
Our goal is to find stability estimates satisfied by the difference function $w$ in some Morrey space. However, as we will see, it is not possible to find $L^2_\alpha$ bounds on $w$ by our method. Instead, we will have to resort to weaker $L^2_\beta$ bounds with the growth rate $\beta > \alpha$ possibly being a function of time. Therefore, we consider a growth rate $\beta > \alpha$  which we will fix later on and take the difference of the equation solved by $u_2$ and the one solved by $u_2$, we see that 
$$
\partial_t w + (u_2 \cdot \nabla) w + (w \cdot \nabla) u_1 + \nabla \Pi (u_2 \otimes w) + \nabla \Pi (w \otimes u_1) = 0.
$$
In order to evaluate the $L^2_\beta$ norm of $w$, we study the energy balance on a ball of radius $R \geq 1$. By multiplying the equation above with $\chi_R w$ and integrating, we obtain
\begin{multline*}
\frac{1}{2} \frac{\rm d}{\dt} \int \chi_R |w|^2 \dx + \int \chi_R (u_2 \cdot \nabla) w \cdot w \dx + \int \chi_R (w \cdot \nabla) u_1 \cdot w \dx \\
+ \int \chi_R \nabla \Pi (u_2 \otimes w) \cdot w \dx + \int \chi_R \nabla \Pi (w \otimes u_1) \cdot w \dx = 0.
\end{multline*}
By noting $I_1, I_2, I_3$ and $I_4$ the last four integrals in the preceding energy balance equation, we obtain
$$
\frac{1}{2} \frac{\rm d}{\dt} \int \chi_R |w|^2 \dx \leq |I_1| + |I_2| + |I_3| + |I_4|,
$$
from which we will attempt to get a differential inequality. Let us look at all four integrals separately.

\medskip

\textbf{First integral} $I_1$. We proceed to an integration by parts in $I_1$. Because both $u_2$ and $w$ are divergence-free, the derivative falls on the cut-off function, so we get
$$
I_1 = \frac{1}{2} \int \chi_R \D (u_2 |w|^2) \dx = - \int \nabla \chi_R \cdot u_2 |w|^2 \dx.
$$
Now,  because of the definition of $\chi_R$, the derivative $\nabla \chi_R$ is uniformly $O(R^{-1})$. And because $\nabla \chi_R$ is supported on the ball $B_{2R}$, we get
\begin{equation*}
\begin{split}
    |I_1| \lesssim \frac{1}{R} \| u_2 \|_{L^\infty(B_{2R})} \int_{B_{2R}} |w|^2 \dx & \lesssim \frac{R^{2 + 2\beta}}{R^{1 - \alpha}} \left( \frac{1}{R^\alpha} \| u_2 \|_{L^\infty (B_{2R})} \right) \left[ \frac{1}{R^{2 + 2 \beta}} \int_{B_{2R}} |w|^2 \dx \right] \\
    & \lesssim \frac{R^{2 + 2\beta}}{R^{1 - \alpha}} \left( \frac{1}{R^\alpha} \| u_2 \|_{L^\infty (B_{2R})} \right) \| w \|_{L^2_\beta}^2.
\end{split}
\end{equation*}
Of course, the quantity in the parenthesis could be bounded by the $L^\infty_\alpha$ norm of $u_2$. However, recall that we do not have $L^\infty_\alpha$ estimates at our disposal for the solution. Instead, Lemma \ref{l:interpolationInequality} shows that $Y_\alpha \subset L^\infty_{(1 + \alpha)/2}$, so that the inequality becomes
\begin{equation}\label{eq:UniquenessI1}
    |I_1| \lesssim \frac{R^{2 + 2 \beta}}{R^{\frac{1 - \alpha}{2}}} \| u_1 \|_{L^\infty_{\frac{1 + \alpha}{2}}} \| w \|^2_{L^2_\beta} \lesssim \frac{R^{2 + 2 \beta}}{R^{\frac{1 - \alpha}{2}}} \| u_1 \|_{Y_\alpha} \| w \|^2_{L^2_\beta}
\end{equation}

\medskip

\textbf{Second integral} $I_2$. As in the case of ``traditional'' Yudovich solutions, this term is delicate because it does not follow from $\omega_1 \in L^\infty$ that we must have $\nabla u_1 \in L^\infty$. In fact, the case of unbounded solutions is somewhat worse because $\omega_1$ is not even $L^p$ for any $p < \infty$, so that Calder\`on-Zygmund theory does not immediately apply to give bounds on $\nabla u_1$. The key in this step of the argument is to write a localized version of the Biot-Savart law: let $\theta, \eta \in \mc D$ be a cut-off functions as in \eqref{eq:cutOff}, and set, as in the proof of existence,
$$
\eta_R(x) := \eta \left( \frac{x}{5R} \right),
$$
so that $\eta_R$ is supported in $B_{10R}$ and $\theta$ in the fixed ball $B_2$. We will use the following decomposition of $\nabla u$, which is similar to the decomposition of the pressure force we have used above.

\begin{lemma}\label{l:BiotSavart}
Consider $u \in L^2_\alpha$. We have the Biot-Savart type law
$$
\nabla u = - \nabla \nabla^\perp \Big( \big( E  \curl\big) * (\eta_R u)\Big)  + \nabla  \nabla^\perp \Big( \big( (\theta E) \curl\big)  * \big( (1 - \eta_R) u \big)\Big)  + \Big(\nabla \Delta \big( (1 - \theta) E \big)\Big) * \big( (1 - \eta_R) u \big).
$$
In the above, all the terms are well-defined as distributions.
\end{lemma}

\begin{proof}[Proof (of the Lemma)] Let $ u \in \mc D (\R^2; \R^2)$. Then 
\[ \begin{split} \nabla u\, &  = -\nabla  ( (-\Delta)^{-1}) \nabla^\perp)  \curl u  \\ &  =  -\nabla  ( (-\Delta)^{-1}) \nabla^\perp)  \curl ( \eta_R u )  -\nabla  ( (-\Delta)^{-1}) \nabla^\perp)  \curl ((1-\eta_R) u )  \\ &  = 
- \nabla \nabla^\perp \Big( \big( E  \curl\big) * (\eta_R u)\Big) 
 + \nabla  \nabla^\perp \Big( \big( (\theta E) \curl\big)  * \big( (1 - \eta_R) u \big)\Big)  \\
 & \qquad \qquad \qquad \qquad \qquad \qquad \qquad \qquad \qquad \qquad + \Big(\nabla \Delta \big( (1 - \theta) E \big)\Big) * \big( (1 - \eta_R) u \big).
 \end{split}
\]
Each term on the right hand side has a unique weak-$(*)$ continuous extension 
to $L^2_\alpha$.
\end{proof}

By using Lemma \ref{l:BiotSavart} in $I_2$, we get a decomposition into three three integrals:
\begin{multline}\label{eq:UniquenessEQ1}
|I_2| \leq \int \chi_R \left| \nabla \nabla^\perp (\theta E) * \curl \big( (1 - \eta_R) u_1 \big) \right| |w|^2 \dx  \\
+ \int \chi_R \left| \nabla \Delta \big( (1 - \theta) E \big) * \big( (1 - \eta_R) u_1 \big) \right| |w|^2 \dx \\
+  \int \chi_R \left| \nabla \nabla^\perp (- \Delta)^{-1} \curl (\eta_R u_1) \right| |w|^2 \dx.
\end{multline}

This decomposition allows us to estimate the integral $I_2$ in a similar way we treated the pressure integral in the proof of Theorem \ref{t:EnergyBound}. We start by looking at the first integral in \eqref{eq:UniquenessEQ1}. The convolution appearing in it is supported in the set
\begin{equation}\label{eq:UniquenessEQ3}
\begin{split}
{\rm supp} \left[ \nabla \nabla^\perp (\theta E) * \curl \big( (1 - \eta_R) u_1 \big) \right] & \subset {\rm supp} (\theta E) + {\rm supp} \big( (1 - \eta_R) u_1 \big) \\
& \subset B_2 + {}^c B_{5R} \subset \{ |x| \geq 3R \},
\end{split}
\end{equation}
and, in particular, the support of the convolution product above is disjoint from the support of $\chi_R$ (which is included in $B_{2R}$). We deduce that the first integral in \eqref{eq:UniquenessEQ1} must be zero:
$$
\int \chi_R \left| \nabla \nabla^\perp (\theta E) * \curl \big( (1 - \eta_R) u_1 \big) \right| |w|^2 \dx = 0.
$$

Next, we study the second integral in \eqref{eq:UniquenessEQ1}. The kernel $\nabla \Delta \big( (1 - \theta) E \big)$ is smooth on $\R^2$ and decays as $O(|x|^{-3})$, so that we have
$$
J(x) := \left| \nabla \Delta \big( (1 - \theta) E \big) * \big( (1 - \eta_R) u_1 \big) (x) \right| \lesssim \int |u_1(y)| \mathds{1}_{|y| \geq 5R} \frac{\dy}{\langle x-y \rangle^3}.
$$
We cut this integral into two parts, according to whether $|x|$ is small compared to $|y|$ (so that $x - y \approx y$) or not (so that $|x-y| \lesssim |x|$). In other words, we have
\begin{equation*}
    J(x) \lesssim \int_{|y| \leq 2|x|} |u_1(y)| \mathds{1}_{|y| \geq 5R} \frac{\dy}{\langle x-y \rangle^3} + \int_{|y| \geq 2 |x|} |u_1(y)| \mathds{1}_{|y| \geq 5R} \frac{\dy}{\langle x-y \rangle^3}.
\end{equation*}
On the one hand, we see that this first integral is supported on the set $\{ 5R \leq |y| \leq 2|x| \}$. Consequently, when $|x| \leq 2R$, it must be zero, so it does not contribute to $I_2$. For the second integral, we use the inequality $|x - y| \geq \frac{1}{2}|y|$, which holds on $\{ |y| \geq 2|x| \}$, to write
\begin{equation*}
    \chi_R(x) J(x) \lesssim \int |u_1(y)| \mathds{1}_{|y| \geq 5R} \frac{\dt}{\langle y \rangle^3}.
\end{equation*}
To evaluate this integral by using the Morrey norm $\| u_1 \|_{L^2_\alpha}$, we recall the dyadic partition of the space \eqref{eq:dydadicSplittingR3Existence} and proceed similarly to \eqref{eq:SumBoundsExistence}. We have
\begin{equation*}
    \begin{split}
        \chi_R(x) J(x) & \lesssim \sum_{k=0}^\infty \int_{A_k(R)} |u_1(y)| \frac{\dy}{\langle y \rangle^3} \\
        & \lesssim \sum_{k=0}^\infty \| u_1 \|_{L^2(B_{2^k R})} (2^k R)^{-3} |A_k(R)|^{1/2} \\
        & \lesssim \| u_1 \|_{L^2_\alpha} \sum_{k = 0}^\infty (2^k R)^{1 + \alpha - 3 + 1} \\
        & \lesssim \frac{1}{R^{1 - \alpha}} \| u_1 \|_{L^2_\alpha}.
    \end{split}
\end{equation*}
By plugging this inequality in the second integral in \eqref{eq:UniquenessEQ1}, we see that it is simply bounded by
\begin{equation}\label{eq:UniquenessEQ5}
\begin{split}
\int \chi_R \left| \nabla \Delta \big( (1 - \theta) E \big) * \big( (1 - \eta_R) u_1 \big) \right| |w|^2 \dx & \lesssim \frac{1}{R^{1 - \alpha}} \| u_1 \|_{L^2_\alpha} \int \chi_R |w|^2 \dx \\
& \lesssim \frac{R^{2 + 2 \beta}}{R^{1 - \alpha}} \| u_1 \|_{L^2_\alpha} \left[ \frac{1}{R^{2 + 2 \beta}} \int_{B_{2R}} |w|^2 \dx \right] \\
& \lesssim\frac{R^{2 + 2 \beta}}{R^{1 - \alpha}} \| u_1 \|_{L^2_{\alpha}} \| w \|_{L^2_\beta}^2.
\end{split}
\end{equation}

Finally, we are left with the third integral in \eqref{eq:UniquenessEQ1}, where we restrict our attention to the localized quantity $\eta_R u_1$. Here, we simply implement the ideas of Yudovich and his use of Calder\`on-Zygmund theory: for any $2 \leq p < \infty$, Theorem \ref{t:czo} provides the bound
\begin{equation*}
\begin{split}
\big\| \nabla \nabla^\perp (- \Delta)^{-1} \curl (\eta_R u_1) \big\|_{L^p} & \lesssim p \| \curl(\eta_R u_1) \|_{L^p} \\
& \lesssim p \| \eta_R \omega_1 \|_{L^p} + p \| \nabla^\perp \eta_R \cdot u_1 \|_{L^p} \\
& \lesssim p R^{2/p} \| \omega_1 \|_{L^\infty} + p \frac{1}{R} \| \nabla \eta_1 \|_{L^\infty} R^{2/p} \| u_1 \|_{L^\infty (B_{10R})} \\
& \lesssim p R^{2/p} \Big( \| \omega_1 \|_{L^\infty} + \frac{1}{R} \| u_1 \|_{L^\infty (B_{10R})} \Big),
\end{split}
\end{equation*}
where we have used the fact that $\eta_R$ and $\nabla \eta_R$ are supported inside the ball $B_{10R}$. In order to use this $L^p$ estimate, we introduce the dual exponent $p'$ defined by $\frac{1}{p} + \frac{1}{p'} = 1$ and invoke Hölder's inequality to obtain
\begin{multline}\label{eq:UniquenessEQ2}
\int \chi_R \left| \nabla \nabla^\perp (- \Delta)^{-1} \curl (\eta_R u_1) \right| |w|^2 \dx \\
\lesssim p R^{2/p} \Big( \| \omega_1 \|_{L^\infty} + \frac{1}{R} \| u_1 \|_{L^\infty (B_{10R})} \Big) \left\| |w|^2 \right\|_{L^{p'}(B_{2R})}
\end{multline}
By using an interpolation inequality in the Lebesgue spaces $L^{2p'} \subset L^{2 \left( 1 - \frac{1}{p} \right)} \cap L^\infty$, we may bound the $L^{p'}$ norm by
$$
\left\| |w|^2 \right\|_{L^{p'}(B_{2R})} = \left\| w \right\|_{L^{2p'}(B_{2R})}^2 \leq \| w \|^{2 \left( 1 - \frac{1}{p} \right)}_{L^2(B_{2R})} \, \| w \|_{L^\infty(B_{2R})}^{\frac{2}{p}},
$$
so that the upper bound in \eqref{eq:UniquenessEQ2} becomes
\begin{equation*}
\begin{split}
& \int \chi_R \left| \nabla \nabla^\perp (- \Delta)^{-1} \curl (\eta_R u_1) \right| |w|^2 \dx  \\
&  \qquad  \lesssim p R^{\frac{2}{p}} \Big( \| \omega_1 \|_{L^\infty} + \frac{1}{R} \| u_1 \|_{L^\infty (B_{10R})} \Big) \| w \|^{2 \left( 1 - \frac{1}{p} \right)}_{L^2(B_{2R})} \, \| w \|_{L^\infty(B_{2R})}^{\frac{2}{p}} \\
&  \qquad  \lesssim p R^{\frac{2}{p}} \Big( \| \omega_1 \|_{L^\infty} + \frac{1}{R} \| u_1 \|_{L^\infty (B_{10R})} \Big) R^{(2 \beta + 2) \left( 1 - \frac{1}{p} \right)} \left[ \frac{1}{R^{2 \beta + 2}} \int_{B_{2R}} |w|^2 \dx \right]^{1 - \frac{1}{p}} \| w \|_{L^\infty (B_{2R})}^{\frac{2}{p}} \\
& \qquad \lesssim p R^{2 + 2 \beta} \| u_1 \|_{Y_\alpha} \| w \|_{L^2_\beta}^{2 \left( 1 - \frac{1}{p} \right)} \left( \frac{1}{R^\beta} \| w \|_{L^\infty(B_{2R})} \right)^{\frac{2}{p}}.
\end{split}
\end{equation*}
At this point, we use the fact that $w = u_2 - u_1$ is the difference of two functions which we know to be bounded in $Y_\alpha \subset L^\infty_{(1 + \alpha)/2}$ (see Lemma \ref{l:interpolationInequality}). By using this to bound the norm $\| w \|_{L^\infty(B_{2R})}$ we obtain
\begin{multline}\label{eq:UniquenessEQ6}
    \int \chi_R \left| \nabla \nabla^\perp (- \Delta)^{-1} \curl (\eta_R u_1) \right| |w|^2 \dx \\
    \lesssim p R^{2 + 2 \beta} \|u_1 \|_{Y_\alpha} \Big( \|u_1 \|_{Y_\alpha} + \| u_2 \|_{Y_\alpha} \Big)^{\frac{2}{p}} R^{\frac{2}{p} \left( \frac{1 + \alpha}{2} - \beta \right)} \| w \|_{L^2_\beta}^{2 \left( 1 - \frac{1}{p} \right)}.
\end{multline}
In comparison to the previous bounds \eqref{eq:UniquenessI1} or \eqref{eq:UniquenessEQ2}, the inequality above has an additional factor of $R^{\frac{2}{p} \left( \frac{1 + \alpha}{2} - \beta \right)}$. As we will see, this creates issues to close the estimates and will have to be dealt with by making an appropriate choice for the value of $\beta$. Leaving these considerations for later and putting the estimates \eqref{eq:UniquenessEQ5} and \eqref{eq:UniquenessEQ6} together, we obtain a bound for the integral $I_2$, namely
\begin{equation}\label{eq:UniquenessI2}
    |I_2| \lesssim \frac{R^{2 + 2 \beta}}{R^{1 - \alpha}} \| u_1 \|_{L^2_\alpha} \| w \|_{L^2_\beta}^2 + p R^{2 + 2 \beta} \|u_1 \|_{Y_\alpha} \Big( \|u_1 \|_{Y_\alpha} + \| u_2 \|_{Y_\alpha} \Big)^{\frac{2}{p}} R^{\frac{2}{p} \left( \frac{1 + \alpha}{2} - \beta \right)} \| w \|_{L^2_\beta}^{2 \left( 1 - \frac{1}{p} \right)}.
\end{equation}
In view of further computations, note that the first summand in the right hand side is bounded by the upper bound \eqref{eq:UniquenessI1} we have already found for $I_1$.

\medskip

\textbf{Third integral} $I_3$. We now study the first of the two pressure integrals. Just as in the case of integral $I_2$ above, we separate the pressure in a localized part and a remainder:
$$
\nabla \Pi (u_2 \otimes w) = \nabla \Pi (\eta_R u_2 \otimes w) + \nabla \Pi \big( (1 - \eta_R) u_2 \otimes w \big).
$$
On the one hand, the first of these pressure terms  gives a contribution to $I_3$ that is
$$
\left| \int \chi_R w \cdot \nabla \Pi (\eta_R u_2 \otimes w ) \dx \right| = \left| \int \nabla \chi_R \cdot w \, \Pi (\eta_R u_2 \otimes w ) \dx \right| \lesssim \frac{1}{R} \| w \|_{L^2(B_{2R})} \big\| \Pi(\eta_R u_2 \otimes w) \big\|_{L^2},
$$
where the operator $\Pi := (- \Delta)^{-1} \nabla^2$ should be understood as a product of Riesz transforms, and is as such $L^2 \tend L^2$ bounded. Consequently, the above yields
\begin{equation*}
\begin{split}
\left| \int \chi_R w \cdot \nabla \Pi (\eta_R u_2 \otimes w ) \dx \right| & \lesssim \frac{1}{R} \| w \|_{L^2(B_{2R})} \big\| \eta_R u_2 \otimes w \|_{L^2} \\
& \lesssim \frac{1}{R} \| u_2 \|_{L^\infty(B_{10R})} \| w \|_{L^2(B_{10R})}^2 \\
& \lesssim \frac{R^{2 \beta + 2}}{R^{\frac{1 + \alpha}{2}}} \| u_2 \|_{L^\infty_{\frac{1 + \alpha}{2}}} \left[ \frac{1}{R^{2 \beta + 2}} \int_{B_{10 R}} |w|^2 \dx \right] \\
& \lesssim \frac{R^{2 \beta + 2}}{R^{\frac{1 - \alpha}{2}}} \| u_2 \|_{Y_\alpha} \| w \|_{L^2_\beta}^2.
\end{split}
\end{equation*}
On the other hand, for the remaining pressure term, we use the decomposition that defines the operator $\nabla \Pi$. Consider $\theta$ be a cut-off function as in \eqref{eq:cutOff}, so that Definition \ref{d:pressureOperator} provides
$$
\nabla \Pi \big( (1 - \eta_R) u_2 \otimes w \big) = \nabla (\theta E) * \nabla^2 : \big( (1 - \eta_R) u_{2} \otimes w \big) + \nabla^3 \big( (1 - \theta) E \big) * \big( (1 - \eta_R) u_2 \otimes w \big).
$$
By reiterating the arguments of \eqref{eq:UniquenessEQ3}, we see that the first summand in the righthand side above is disjointly supported from $\chi_R$, and therefore can be ignored. We then focus on the second summand and define, for notational convenience,
$$
J(x) := \nabla^3 \big( (1 - \theta) E \big) * \big( (1 - \eta_R) u_2 \otimes w \big) (x).
$$
Then, we have, by writing explicitly the value of the convolution product at a point $x \in \R^2$ and decomposing the plane as $\R^2 = \{ |y| \leq 2|x| \} \cup \{ |y| \geq 2|x| \}$,
\begin{equation}\label{eq:UniquenessEQ4}
\begin{split}
|J(x)| & \lesssim \int \big| u_2 \otimes w (y) \big| \frac{\mathds{1}_{|y| \geq 5R}}{\langle x - y \rangle^{3}} \dy \\
& \lesssim \int_{|y| \leq 2|x|} \big| u_2 \otimes w (y) \big| \frac{\mathds{1}_{|y| \geq 5R}}{\langle x - y \rangle^{3}} \dy + \int_{|y| \geq 2 |x|} \big| u_2 \otimes w (y) \big| \frac{\mathds{1}_{|y| \geq 5R}}{\langle x - y \rangle^{3}} \dy.
\end{split}
\end{equation}
The first of these two integrals is defined on the set $5 R \leq |y| \leq 2 |x|$, and therefore must vanish if $|x| \leq 5/2R$. Consequently, when $J(x)$ is multiplied by $\chi_R(x)$, this integral has no bearing whatsoever on the computations, and we may focus on the integral on $|y| \geq 2|x|$. By triangle inequality, we see that
$$
\frac{1}{\langle x - y \rangle^3} \lesssim \frac{1}{\langle y \rangle^3}
$$
whenever $|y| \geq 2 |x|$, so we are reduced to estimating an integral whose integrand depends only on $y$. Here, let us recall from \eqref{eq:dydadicSplittingR3Existence} the dyadic splitting $\bigsqcup_k A_k(R)$ of the space we used for the proof of existence. Using it to decompose the last integral in \eqref{eq:UniquenessEQ4}, we have, by the Cauchy-Schwarz inequality,
\begin{equation*}
    \begin{split}
        \big| \chi_R (x) J(x) \big| & \lesssim \sum_{k=0}^\infty \int_{A_k(R)} \big| u_2 \otimes w (y) \big| \frac{\dy}{\langle y \rangle^3} \\
        & \lesssim \sum_{k=0}^\infty \left( \int_{B_{2^{k+1}R}} |w|^2 \right)^{1/2} \left( \int_{B_{2^{k+1}R}} |u_2|^2  \right)^{1/2} (2^k R)^{-3}.
    \end{split}
\end{equation*}
We now try to bound each of these integrals by using the $L^2_\beta$ and $L^2_\alpha$ norms of $w$ and $u_2$ respectively. The above inequality becomes:
\begin{equation}\label{eq:AlphaBetaPressure}
    \begin{split}
        \big| \chi_R (x) J(x) \big| & \lesssim \| w \|_{L^2_\beta} \| u_2 \|_{L^2_\alpha} \sum_{k=0}^\infty \left( 2^{k} R \right)^{\alpha + \beta - 1}\\
        & \lesssim \frac{1}{(1 - \alpha - \beta) R^{1 - \alpha - \beta}} \| u_2 \|_{L^2_\alpha} \| w \|_{L^2_\beta}.
    \end{split}
\end{equation}
The computation above is valid provided the sum converges. This means that we are restricted to growth rates such that $\beta < 1 - \alpha$. Under this condition, we have shown that the integral involving $J$ can be bounded by
\begin{equation}\label{eq:UniquenessI3}
    \begin{split}
        \left| \int \chi_R w \cdot J \dx \right| & \lesssim D(\alpha, \beta) \frac{1}{R^{1 - \alpha - \beta}} \| w \|_{L^2_{\beta}} \| u_2 \|_{L^2_\alpha} \int \chi_R |w| \dx \\
        & \lesssim D(\alpha, \beta) \frac{R^{2 + 2 \beta}}{R^{1 - \alpha}} \| u_2 \|_{L^2_\alpha} \| w \|_{L^2_\beta}^2.
    \end{split}
\end{equation}
The constant $D(\alpha, \beta)$ appearing in the inequality above is bounded by $C(\alpha) (1 - \alpha - \beta)^{-1}$.

\medskip

\textbf{Fourth integral} $I_4$. The fourth integral can be estimated exactly in the same way as $I_3$. We have
\begin{equation}\label{eq:UniquenessI4}
|I_4| \lesssim D(\alpha, \beta) \frac{R^{2 + 2 \beta}}{R^{1 - \alpha}} \| u_1 \|_{L^2_\alpha} \| w \|_{L^2_\beta}^2.
\end{equation}
The constant $D(\alpha, \beta)$ is defined just as above.

\medskip

\textbf{Summary of the estimates}. We put together the estimates \eqref{eq:UniquenessI1}, \eqref{eq:UniquenessI2}, \eqref{eq:UniquenessI3} and \eqref{eq:UniquenessI4} in order to obtain a differential inequality. We have, by keeping in mind that $R \geq 1$,
\begin{equation}\label{eq:StabiliyRSummary}
    \begin{split}
        \frac{\rm d}{\dt} \int \chi_R |w|^2 & \lesssim D(\alpha, \beta) R^{2 + 2 \beta} \| w \|_{L^2_\beta}^2 \left( \frac{1}{R^{1 - \alpha}} \| u_1, u_2 \|_{L^2_\alpha} + \frac{1}{R^{\frac{1 - \alpha}{2}}} \| u_1, u_2 \|_{Y_\alpha} \right) \\
        & \qquad \qquad \qquad \qquad + p R^{2 + 2 \beta} \| u_1 \|_{Y_\alpha} \| u_1, u_2 \|_{Y_\alpha}^{2/p} R^{\frac{2}{p} \left( \frac{1 + \alpha}{2} - \beta \right)} \| w \|_{L^2_\beta}^{2 \left( 1 - \frac{1}{p} \right)} \\
        & \lesssim D(\alpha, \beta) R^{2 + 2 \beta} K(t) \| w \|_{L^2_\beta}^2 + p R^{2 + 2 \beta} K(t)^{1 + 2/p} R^{\frac{2}{p} \left( \frac{1 + \alpha}{2} - \beta \right)} \| w \|_{L^2_\beta}^{2 \left( 1 - \frac{1}{p} \right)},
    \end{split}
\end{equation}
where where we have defined the quantity (see Definition \ref{d:YudovichSpace} for the space $Y_\alpha$)
\begin{equation*}
    K(t) = \sup_{[0, t]} \| u_1, u_2 \|_{Y_\alpha}.
\end{equation*}
If we are to close the estimates and find $L^2_\beta$ bounds on the difference function $w$, then a supremum over all radii $R \geq$ has to be taken, in addition to the usual Yudovich argument. However, the second term in the upper bound above involves a $R^{\frac{2}{p} \left( \frac{1 + \alpha}{2} - \beta \right)}$ factor which introduces additional growth if $\beta$ is not large enough. Here, two things may happen. If $\alpha < 1/3$, a simple choice of $\beta$ will lead to the sought estimate for $w$. In the harder case where we may have $1/3 \leq \alpha < 1/2$, we will have to resort to letting $\beta$ depend on time to prove uniqueness. In both cases, we will have to resort to interpolation arguments to end the proof.

\medskip

\textbf{Easier case:} we first assume that $\alpha < 1/3$. Then it is possible to fix a constant $\beta$ such that
\begin{equation}\label{eq:betaConstantInequalities}
     \frac{1 + \alpha}{2} \leq \beta < 1 - \alpha.
\end{equation}
In particular, the pressure estimates \eqref{eq:AlphaBetaPressure} hold and the factor $R^{\frac{2}{p} \left( \frac{1 + \alpha}{2} - \beta \right)}$ is bounded as $R \geq 1$. In that case, inequality \eqref{eq:StabiliyRSummary} reduces to
\begin{equation*}
    \frac{\rm d}{\dt} \left\{ \frac{1}{R^{2 + 2 \beta}} \int \chi_R |w|^2 \right\} \lesssim K(t)\| w \|_{L^2_\beta}^2 + p K(t)^{1 + 2/p} \| w \|_{L^2_\beta}^{2 \left( 1 - \frac{1}{p} \right)}.
\end{equation*}
We then integrate on the time interval $[0, T]$ and take the supremum over all values of $R \geq 1$. This leads to an integral inequality on the $L^2_\beta$ norm:
\begin{equation*}
    \| w(T) \|_{L^2_ \beta}^2 \leq \| w_0 \|_{L^2_\beta}^2 + \int_0^T K(t) \| w \|_{L^2_\beta}^2 \dt + p K(T) \int_0^T \| w \|_{L^2_\beta}^2 \left( \frac{K(T)^2}{\| w \|_{L^2_\beta}^2} \right)^{1/p} \dt.
\end{equation*}
This inequality holds for all values of $p \in [2, \infty[$. Therefore, we may optimize it. Set
\begin{equation*}
    p := \log \left( \frac{18 K(T)^2}{\| w \|_{L^2_\beta}^2} \right).
\end{equation*}
Note that the norm $\| w \|_{L^2_\beta}$ is always bounded by the quantity $2K(T)$, so that $p$ as defined above is always at least $p \geq 2$, since $18 \geq 2 e^2$. By plugging this value of $p$ inside the integral inequality, we find that
\begin{equation}\label{eq:GGronwall}
    \| w \|_{L^2_\beta}^2 \leq Q(T) := \| w_0 \|_{L^2_\beta}^2 + C K(T) \int_0^T \| w \|_{L^2_\beta}^2 \left[ 1 + \log \left( \frac{18 K(T)^2}{\| w \|_{L^2_\beta}^2} \right) \right] \dt.
\end{equation}
In order to deduce an explicit upper bound from the above, we follow the computations given in (for instance) the proof of Lemma 3.4 of \cite{BCD}, which we reproduce here for the reader's convenience. Define the functions
\begin{equation*}
    \mu(r) := CK r \left[ 1 + \log \left( \frac{18K^2}{r} \right) \right] \qquad \text{and} \qquad \mc M(r) := \int_r^{2K} \frac{1}{\mu} = \frac{1}{CK} \log \left[ 1 + \log \left( \frac{18K^2}{r} \right) \right] + {\rm Cst}.
\end{equation*}
The function $Q(T)$ is defined so that its derivative is $Q'(T) = \mu( \| w \|_{L^2_\beta}^2 )$. And because the function $r \mapsto \mu(r)$ is increasing, inequality \eqref{eq:GGronwall} implies that the quantity $Q(T)$ satisfies the differential inequality $Q'(T) \leq \mu (Q(T))$. As a consequence, we have
\begin{equation*}
    -\frac{\rm d}{{\rm d}T} \mc M \big( Q(T) \big) = \frac{Q'(T)}{\mu \big( Q(T) \big)} \leq 1,
\end{equation*}
so that integrating provides the inequality
\begin{equation*}
    \log \left[ 1 + \log \left( \frac{18 K(T)^2}{\| w_0 \|_{L^2_\beta}^2} \right) \right] \leq CKT + \log \left[ 1 + \log \left( \frac{18 K(T)^2}{\| w \|_{L^2_\beta}^2} \right) \right],
\end{equation*}
and exponentiating twice gives the desired inequality
\begin{equation*}
    \big\| w(T) \big\|_{L^2_\beta}^2 \leq \| w_0 \|_{L^2_\beta}^{2 e^{-TCK(T)}} \left( 18 K(T)^2 \right)^{1 - e^{-TCK(T)}}.
\end{equation*}
In particular, this shows that the initial data to solution map is locally Hölder in the $L^2_\beta \longrightarrow L^2_\beta$ topology, with a Hölder exponent depending on the norm of the initial data  and on time. In order to obtain $L^2_\gamma$ continuity for any $\alpha < \gamma < 1 - \alpha$, we resort to an interpolation argument. 

\medskip

Consider a growth rate $\alpha < \gamma < 1 - \alpha$ and a $\theta \in ]0, 1[$ such that $\gamma = \theta \beta + (1 - \theta) \alpha$. Then the interpolation inequality of Lemma \ref{l:interpolationMorrey} provides
\begin{equation*}
    \begin{split}
        \| w(T) \|_{L^2_\gamma} & \leq \| w \|_{L^2_\beta}^\theta \| w \|_{L^2_\alpha}^{1 - \theta} \\
        & \lesssim K(T)^{1 - \theta + \theta(1 - e^{-CTK(T)})} \| w_0 \|_{L^2_{\beta}}^{\theta e^{-TCK(T)}} \\
        & = K(T)^{1 - \theta e^{-CTK(T)}} \| w_0 \|_{L^2_{\beta}}^{\theta e^{-TCK(T)}}.
    \end{split}
\end{equation*}
This proves that the initial data to solution map is continuous (indeed locally Hölder) in the $L^2_\beta \longrightarrow L^2_\gamma$ topology, under the bounded vorticity constraint. However, this inequality is slightly weaker than \eqref{eq:localHolder}, because of the presence of the interpolation coefficient $\theta$. We will see that we can get rid of it by using more involved estimates.

\medskip

\textbf{Harder case:} we no longer assume that $\alpha \geq 1/3$. Fix a $T > 0$. As explained above, the problem in this case is that the factor $R^{\frac{2}{p} \left( \frac{1 + \alpha}{2} - \beta \right)}$ can create a growth in the upper bound which prevents us from closing the estimates, and here we cannot fix a $\beta$ such that \eqref{eq:betaConstantInequalities} holds, because we may have $\alpha \geq 1/3$. This time, we first start by optimizing the value on the Lebesgue exponent before taking the supremum over all radii. We therefore choose for $p$ a value which depends on $R \geq 1$. We set, for any time $t \in [0, T]$,
\begin{equation*}
    p(R) := \log \left( \frac{18 K(T)^2 R^{\left( \frac{1 + \alpha}{2} - \beta\right)_+ }}{\| w(t) \|_{L^2_\beta}^2} \right).
\end{equation*}
Because we consider radii $R \geq 1$, the power of $R$ in the logarithm is at least $1$, since the exponent $\left( \frac{1 + \alpha}{2} - \beta\right)_+$ is nonnegative. Once again, this value of the exponent is larger than $p(R) \geq 2$. By plugging this into the differential inequality \eqref{eq:StabiliyRSummary}, we get the estimate
\begin{equation}\label{eq:logRuniquenessDiff}
    \frac{\rm d}{\dt} \int \chi_R |w|^2 \lesssim D(\alpha, \beta) R^{2 + 2 \beta} K(T) \| w(t) \|_{L^2_\beta}^2 \left[ 1 + \log \left( \frac{18 K(T)^2 R^{\frac{1 + \alpha}{2} - \beta}}{\| w(t) \|_{L^2_\beta}^2} \right) \right].
\end{equation}
This vastly improves the growth rate of the upper bound with respect to $R$, but it still is logarithmic. In order to get rid of that growth, we will let the rate $\beta$ depend on time. More precisely, fix a $\gamma > \alpha$ such that $1 - \alpha - \gamma > 0$ and define
\begin{equation*}
    \beta_t = \gamma + At \qquad \text{for } 0 \leq t \leq t_1
\end{equation*}
for some constant $A > 0$ which we will fix later on and some time $t_1$ which has to be small enough for the inequality $\beta_s < 1 - \alpha$ to hold for all $t \leq t_1$. Then, by integrating \eqref{eq:logRuniquenessDiff} on the time interval $[0, t] \subset [0, t_1] \subset [0, T]$ and multiplying by $R^{-2-2\beta_t}$, we obtain 
\begin{equation}\label{eq:HarderIntInequality}
    \begin{split}
        \frac{1}{R^{2 + 2 \beta_t}} & \int_{B_R} |w(t)|^2 \\
        & \leq \|w_0 \|_{L^2_\gamma}^2 + D(\alpha, \beta_{t_1}) K(T) \int_0^t R^{2A(s-t)} \| w \|_{L^2_{\beta_s}}^2 \left[ 1 + \log \left( \frac{18 K(T)^2 R^{\frac{1 + \alpha}{2} - \beta_s}}{\| w \|_{L^2_{\beta_s}}^2} \right) \right] {\rm d}s \\
        & \lesssim \| w_0 \|_{L^2_\gamma}^2 + D(\alpha, \beta_{t_1}) K(T) \int_0^t e^{2A(s-t)\log(R)} \| w \|_{L^2_{\beta_s}}^2 \left( \frac{1 + \alpha}{2} - \beta_s \right) \log (R) {\rm d} s \\
        & \qquad \qquad \qquad \qquad \qquad + D(\alpha, \beta_{t_1}) K(T) \int_0^t \| w \|_{L^2_{\beta_s}}^2 \left[ 1 + \log \left( \frac{18 K(T)^2}{\| w \|_{L^2_{\beta_s}}^2} \right) \right] {\rm d} s \\
        & \leq \| w_0 \|_{L^2_\gamma}^2 + D(\alpha, \beta_{t_1}) \frac{K(T)}{2A} \left( \frac{1+\alpha}{2} - \gamma \right) \sup_{[0, t]} \left( \| w \|_{L^2_{\beta_s}}^2 \right) \\
        & \qquad \qquad \qquad \qquad \qquad + D(\alpha, \beta_{t_1}) K(T) \int_0^t \| w \|_{L^2_{\beta_s}}^2 \left[ 1 + \log \left( \frac{18 K(T)^2}{\| w \|_{L^2_{\beta_s}}^2} \right) \right] {\rm d} s.
    \end{split}
\end{equation}
In particular, this upper bound no longer depends on the radius. By taking the supremum over $R \geq 1$ on the lefthand side, we find an estimate on $\| w \|_{L^2_{\beta_t}}$. We now focus on the first term in the righthand side. By remembering that the constant $D(\alpha, \beta_{t_1})$ is of the form $C(\alpha)(1 - \alpha - \beta_{t_1})^{-1}$, we see that it is bounded by
\begin{equation}\label{eq:absorbedTerm}
    D(\alpha, \beta_{t_1}) \frac{K(T)}{2A} \left( \frac{1+\alpha}{2} - \gamma \right) \sup_{[0, t]} \left( \| w \|_{L^2_{\beta_s}}^2 \right) \leq \frac{C(\alpha)K(T)}{A(1 - \alpha - \gamma - At_1)} \sup_{[0, t]} \left( \| w \|_{L^2_{\beta_s}}^2 \right).
\end{equation}
By taking $A > 0$ sufficiently large and $t_1 > 0$ sufficiently small, it is possible to make the fraction in the upper bound small. However, we must also insure that we also have $1 - \alpha - \beta_{t_1} > 0$, so as to not destroy the pressure estimates. We therefore fix
\begin{equation*}
    t_1 := \frac{(1 - \alpha - \gamma)^2}{16C(\alpha)K(T)} \qquad \text{and} \qquad A := \frac{8C(\alpha)K(T)}{1 - \alpha - \gamma}.
\end{equation*}
In this case, the fraction in the righthand side of \eqref{eq:absorbedTerm} is exactly $1/4$ and we also insure that the pressure remains well-defined since
\begin{equation*}
    1 - \alpha - \beta_{t_1} = \frac{1}{2}(1 - \alpha - \gamma).
\end{equation*}
With the previous choice of $A$ and $t_1$, we may absorb the term in \eqref{eq:absorbedTerm} in the lefthand side of inequality \eqref{eq:HarderIntInequality} so that it becomes 
\begin{equation*}
    \sup_{[0, t]} \left( \| w \|_{L^2_{\beta_t}}^2 \right) \lesssim \| w_0 \|_{L^2_\gamma}^2 + K(T) \int_0^t \| w \|_{L^2_{\beta_s}}^2 \left[ 1 + \log \left( \frac{18 K(T)^2}{\| w \|_{L^2_{\beta_s}}^2} \right) \right] {\rm d}s,
\end{equation*}
and we may rewrite the computations we did in the case $\alpha < 1/3$ in order to deduce a $L^2_\gamma \longrightarrow L^2_{\beta_{t_1}}$ stability estimate on the short time interval $[0, t_1]$, namely
\begin{equation*}
    \begin{split}
        \sup_{[0, t_1]} \left( \| w \|_{L^2_{\beta_{t}}} \right) & \lesssim \| w_0 \|_{L^2_\gamma}^{e^{-t_1CK(T)}} K(T)^{1 - e^{-t_1 C K(T)}} \\
        & \lesssim \| w_0 \|_{L^2_\gamma}^{e^{-C}} K(T)^{1 - e^{-C}}.
    \end{split}
\end{equation*}
The last inequality above is due to the fact that the quantity $t_1 K(T)$ only depends on $(\alpha, \gamma)$. Therefore, all constants in the above (including implicit constants) only depend on $(\alpha, \gamma)$. In order to recover a $L^2_\gamma \longrightarrow L^2_\gamma$ estimate, we interpolate: fix a $\theta \in ]0, 1[$ such that $\gamma = \theta \beta_{t_1} + (1 - \theta)\alpha$ (in particular, $\theta$ only depends on $\alpha$ and $\gamma$) and invoke interpolation Lemma \ref{l:interpolationMorrey} to write
\begin{equation*}
    \begin{split}
        \| w(t_1) \|_{L^2_\gamma} & \lesssim \| w_0 \|_{L^2_\gamma}^{\theta e^{-C}} \| w_0 \|_{L^2_\alpha}^{1 - \theta} K(T)^{\theta (1 - e^{-C})}\\
        & \lesssim \| w_0 \|_{L^2_\gamma}^{\theta e^{-C}} K(T)^{1 - \theta e^{-C}}.
    \end{split}
\end{equation*}
Finally, we iterate this inequality step by step in time to get an estimate on $w(T)$. Consider an integer $n \geq 1$ such that $n t_1 \geq T$, or in other words, $n \geq C(\alpha, \gamma) TK(T)$. Then the previous inequality provides
\begin{equation*}
    \| w(T) \|_{L^2_\gamma} \lesssim \| w_0 \|_{L^2_\gamma}^{(\theta e^{-C})^n} K(T)^{1 - (\theta e^{-C})^n}.
\end{equation*}
By setting
\begin{equation*}
    - C_1 = C(\alpha, \gamma) \log(\theta e^{-C}) < 0,
\end{equation*}
which only depends on $\alpha$ and $\gamma$, we finally obtain the global in time estimate we were seeking:
\begin{equation*}
    \| w(T) \|_{L^2_\gamma} \lesssim \| w_0 \|_{L^2_\gamma}^{e^{-C_1TK(T)}} K(T)^{1 - e^{-C_1TK(T)}}.
\end{equation*}
This ends proving the uniqueness of solutions as well as inequality \eqref{eq:localHolder} for all values of $0 \leq \alpha < \frac{1}{2}$.

\medskip

\textbf{Continuity of the initial data to solution map.} Let $\epsilon, \delta > 0$ and consider two initial data $u_0 \in L^2_{\alpha, 0}$ and $v_0 \in L^2_\alpha$ such that $\| u_0 - v_0 \|_{L^2_\alpha} \leq \delta$ and $\| \curl(u_0, v_0) \|_{L^\infty} \leq B$. Denote by $u(t), v(t)$ the corresponding solutions evaluated at time $t > 0$. According to Theorem \ref{t:EnergyBound} (and more precisely equation \eqref{eq:IntIneq1}), there exists a $R_1 > 0$ such that for any $R \geq R_1$, we have
\begin{equation}\label{eq:FlowL2A0Estimate}
    \frac{1}{R^{1 + \alpha}} \left( \big\| u(t) \big\|_{L^2(B_R)} + \big\| v(t) \big\|_{L^2(B_R)} \right) \leq \frac{2}{R^{1 + \alpha}} \left( \| u_0 \|_{L^2(B_R)} + \| v_0 \|_{L^2(B_R)}\right).
\end{equation}
The radius $R_1$ only depends on the time $t > 0$, on $B$ and on $\| u_0, v_0 \|_{L^2_\alpha}$. Because $u_0 \in L^2_{\alpha, 0}$, we in particular have the convergence $R^{-(1 + \alpha)} \| u(t) \|_{L^2(B_R)} \longrightarrow 0$ as $R \rightarrow \infty$, and this shows that $u(t) \in L^2_{\alpha, 0}$. As a consequence, the initial data to solution map is well defined as a $Y_\alpha \cap L^2_{\alpha, 0} \longrightarrow Y_\alpha \cap L^2_{\alpha, 0}$ map.

\medskip

Because $u_0 \in L^2_{\alpha, 0}$, we may fix a radius $R_2 > 0$ only depending on $R_2$ such that the quantity $R^{-(1 + \alpha)} \| u_0 \|_{L^2(B_R)}$ is smaller than $\epsilon$. The inequality $\| u_0 - v_0 \|_{L^2_\alpha} \leq \delta$ and the definition of $L^2_\alpha$ imply that $R^{-(1 + \alpha)} \| v_0 \|_{L^2(B_R)} \leq \epsilon + \delta$ for all $R \geq R_2$. By equation \eqref{eq:FlowL2A0Estimate}, this implies in turn that for $R \geq R_0 := \max \{ R_1, R_2 \}$, we have
\begin{equation*}
    \frac{1}{R^{1 + \alpha}} \left( \big\| u(t) \big\|_{L^2(B_R)} + \big\| v(t) \big\|_{L^2(B_R)} \right) \leq 2(\epsilon + \delta),
\end{equation*}
and the radius $R_0$ only depends on $\| u_0 \|_{L^2_\alpha} + \| v_0 \|_{L^2_\alpha} \leq \| u_0 \|_{L^2_\alpha} + \delta$ and on $u_0$.

\medskip

Now fix a $\gamma > \alpha$. Points \textit{(i)} and \textit{(ii)} of Theorem \ref{t:UniquenessYudovich}, which we have already proven show that
\begin{equation*}
    \big\| u(t) - v(t) \big\|_{L^2_\gamma} \leq A \| u_0 - v_0 \|_{L^2_\gamma}^\theta,
\end{equation*}
where $A$ and $\theta$ are functions of $\big( t, B, \| u_0, v_0 \|_{L^2_\alpha} \big)$. Combining this with what we have already inferred above, we get that
\begin{equation*}
    \frac{1}{R^{1 + \alpha}} \big\| u(t) - v(t) \big\|_{L^2(B_R)} \lesssim
    \begin{cases}
        R_0^{\gamma - \alpha} A \| u_0 - v_0 \|_{L^2_\gamma}^\theta \quad \text{if } R \leq R_0 \\
        \epsilon + \delta \quad \text{if } R \geq R_0.
    \end{cases}
\end{equation*}
This shows that the norm $\big\| u(t) - v(t) \big\|_{L^2_\alpha}$ can be made smaller than $2 \epsilon$ by taking $\delta$ smaller than some $\eta > 0$ which only depends on $\epsilon$ and $R_0$, \textsl{i.e.} on $\epsilon$ and on $u_0$. This proves the sought continuity property (see Remark \ref{r:VanishPerturbation}).

\end{proof}

\section{Representation Formula for the Pressure}\label{s:PressureFormula}

In this final section, we explain the role of the far-field condition $u(t) - u(0) \in \mc S'_h$. We will see that a $Y_\alpha$ solution of the Euler equations is a Yudovich solution if and only if it satisfies this far-field condition (see Theorem \ref{t:RepresentationPressure}).

\subsection{The Issue: not every Regular Solution is a Yudovich Solution}

We study the question of whether a general $Y_\alpha$ weak solution of the Euler equations is a Yudovich solution. As we have explained in the introduction, infinite energy solutions create a difficulty here, as it is not obvious that a solution $(u, \pi)$ of the Euler equations
\begin{equation*}
    \begin{cases}
        \partial_t u + \D (u \otimes u) + \nabla \pi = 0\\
        \D(u) = 0
    \end{cases}
\end{equation*}
must have a pressure that satisfies $\nabla \pi = \nabla \Pi (u \otimes u)$, where $\nabla \Pi$ is the pressure operator from Definition \ref{d:pressureOperator}. In fact, the generalized Galilean invariance of the equations \eqref{ieq:GenGalileo} proves that it is not the case: for any Yudovich solution of the Euler equations $u \in L^\infty_{\rm loc} (Y_\alpha)$ and any smooth $g : \R \rightarrow \R^2$, the functions $(v, p)$ from equation \eqref{ieq:GenGalileo} define a solution of the Euler equations with $\nabla p \neq \nabla \Pi (v \otimes v)$. This phenomenon is very much linked to non-uniqueness of the solutions: while Yudovich solutions are unique by Theorem \ref{t:UniquenessYudovich}, it is possible to pick the function $g$ with its support away from $t=0$, and hence to construct infinitely many different solutions that share the same initial data.

\medskip

These considerations show that if any kind of uniqueness theorem is to be achieved, we must find a way to characterize solutions $(u, \pi)$ whose pressure force is given by Definition \ref{d:pressureOperator}, in other words solutions so that $\nabla \pi = \nabla \Pi (u \otimes u)$. 

\medskip

We point out that the results in this section hold in a much more general setting than solutions of the 2D Euler equations with bounded vorticity. By adopting a distributional generalization of the pressure operator $\nabla \Pi (u \otimes u)$ in any dimension $d \geq 2$ and for $L^2_\alpha$ solutions, we are able to find a necessary and sufficient condition for a solution $(u, \pi)$ of the Euler equations to satisfy $\nabla \pi = \Pi(u \otimes u)$.

\begin{defi}\label{d:GenPressureOperator}
    Let the dimension of the space be $d \geq 2$. Consider $\alpha < \frac{1}{2}$, cut-off functions $\chi, \theta \in \mc D$ as in \eqref{eq:cutOff}, and $u, v \in L^2_\alpha$. Let $\psi \in \mc S$ be the Schwartz function such that $\what{\psi}(\xi) = \chi(\xi)$. We define the distribution $\nabla \bar{\Pi}(u \otimes v) \in \mc S'$ by the following formula:
    \begin{equation*}
        \nabla \bar{\Pi}(u \otimes v) := \big( {\rm Id} - \chi(D) \big) \nabla (- \Delta)^{-1} \nabla^2 : (u \otimes v) + \Gamma * (u \otimes v),
    \end{equation*}
    where the kernel $\Gamma$ of the operator $\nabla \chi(D) (- \Delta)^{-1} \nabla^2 :$ is defined by
    \begin{equation*}
        \Gamma := \nabla^3 \psi * (\theta E) + \psi * \nabla^3 \big( (1 - \theta)E \big).
    \end{equation*}
    The operator $\nabla \bar{\Pi}$ is independent of the choice of cut-off functions $\chi, \theta \in \mc D$. The kernel $\Gamma$ depends on the choice of $\chi$, but not on that of $\theta$.
\end{defi}

Several remarks are necessary to to explain this definition.

\begin{rmk}
    Even though it is not obvious by the way the operator $\nabla \bar{\Pi}$, we will see that this definition makes sense for all functions $u, v \in L^2_\alpha$, without any additional regularity requirements, as long as $\alpha < \frac{1}{2}$.
\end{rmk}

\begin{rmk}
    It is quite clear that if $u, v \in Y_\alpha$, then the two pressure operators $\nabla \Pi (u \otimes v)$ and $\nabla \bar{\Pi} (u \otimes v)$ coincide: in that case $\Gamma * (u \otimes v) = \chi(D) \nabla \Pi(u \otimes u)$ while $\big( {\rm Id} - \chi(D) \big) \nabla (- \Delta)^{-1} \nabla^2 : (u \otimes v)$ is equal to $\big( {\rm Id} - \chi(D) \big) \nabla \Pi (u \otimes v)$. However, by Remark \ref{r:pRegularityPressure}, we see that the quantity $\nabla \Pi(u \otimes v)$ from Definition \ref{d:pressureOperator} only makes sense when $u, v \in L^p_{\rm loc}$ for some $p > 2$. This is why Definition \ref{d:GenPressureOperator} is a necessary generalization of Definition \ref{d:pressureOperator} if we want to deal with $L^2_\alpha$ solutions.
\end{rmk}

\subsection{A Far-Field Condition for Yudovich Solutions}

In this subsection, we state and prove the main result of this paragraph: necessary and sufficient conditions for a weak solution $u \in C^0(L^2_\alpha)$ of the Euler equations to have a pressure force given by $F_\pi = \nabla \bar{\Pi}(u \otimes u)$. In particular, this allows to determine whether a $Y_\alpha$ solution in fact is a Yudovich solution. This result, Theorem \ref{t:RepresentationPressure}, is in the spirit of the results of \cite{Cobb1}, although the proof is much more technical due to the $L^2_\alpha$ framework.

It should be noted that Theorem \ref{t:RepresentationPressure} holds in any dimension $d \geq 2$. It therefore has potential applications beyond the scope of Yudovich solutions for the 2D Euler equations.

\begin{thm}\label{t:RepresentationPressure}
    Let $T > 0$, $0 \leq \alpha < \frac{1}{2}$, and $d \geq 2$. Consider a weak solution $u \in C^0\big([0, T[ ; L^2_\alpha(\R^d)\big)$ of the Euler equations associated with the divergence-free initial datum $u_0 \in L^2_\alpha(\R^d)$ according to Definition \ref{d:WeakSolution}. Then the following conditions are equivalent:
    \begin{enumerate}[(i)]
        \item the pressure force $F_\pi$ is given by the integral formula $F_\pi = \nabla \bar{\Pi}(u \otimes u)$, which holds in $\mc D'([0, T[ \times \R^d)$;
        \item for all $t \geq 0$, we have $u(t) - u(0) \in \mc S'_h$ and $u(0) = u_0$;
        \item the pressure force is continuous with respect to time $F_\pi \in C^0([0, T[ ; \mc S')$ and satisfies $F_\pi (t) \in \mc S'_h$ for all $t \geq 0$.
    \end{enumerate}
\end{thm}

\begin{rmk}
    The regularity requirements of Theorem \ref{t:RepresentationPressure} deserve some attention. On the one hand, a very natural assumption would be that $u$ have $W^{1, \infty}$ regularity with respect to the time variable in some low regularity space: in fact, Yudovich solutions have this time regularity due to the equation they solve:
    \begin{equation*}
        \partial_t u = - \D (u \otimes u) - \nabla \Pi (u \otimes u).
    \end{equation*}
    However, as we have seen in Remark \ref{r:timeIrregularity}, weak solutions in general do not possess automatic regularity in the time variable, even in $\mc D'$ topology. On the other hand, it seems pointless to go below $C^0$ time regularity, in the light of the existence of paradoxical solutions (that dissipate kinetic energy) \cite{DS}. Overall, time continuity seems to be a reasonable regularity requirement for our result.
\end{rmk}

\begin{proof}
    The key point of the proof is to show that the integral representation of the pressure $\nabla \bar{\Pi} (u \otimes u)$ is an element of $\mc S'_h$ for almost every time (see Definition \ref{d:Chemin}). In other words, we wish to make sure that
    \begin{equation}\label{eq:PressureConvergence}
        \chi(\lambda D) \nabla \bar{\Pi}(u \otimes u) \tend_{\lambda \rightarrow \infty} 0 \qquad \text{in } \mc S'.
    \end{equation}
    for every $u \in L^2_\alpha$. We will start by seeking a kernel estimate for the operator $\chi(\lambda D) \nabla \bar{\Pi}$ before applying them to the convergence \eqref{eq:PressureConvergence}. Finally, in a third and last step of the proof, we will make use of this in the Euler system.

    \medskip

    \textbf{STEP 1: kernel estimates}. We start by finding bounds for the low frequency part of the kernel of the $\nabla \Pi$ operator. This is contained in the proposition below.

    \begin{prop}\label{p:KernelEstimate}
        Consider $\chi \in \mc D$ a cut-off function as in \eqref{eq:cutOff} and let $\psi_\lambda \in \mc S$ be such that $\what{\psi_\lambda}(\xi) = \chi(\lambda \xi)$. Then the kernel of the operator $\chi(\lambda D) \nabla \bar{\Pi}$, namely
        \begin{equation}\label{eq:KernelLambda}
            \Gamma_\lambda := \nabla^2 \psi_\lambda * \nabla (\theta E) + \psi_\lambda * \nabla^3 \big( (1 - \theta) E \big)
        \end{equation}
        satisfies the inequality
        \begin{equation*}
            |\Gamma_\lambda (x)| \lesssim \frac{1}{\lambda^{d+1}} \left\langle \frac{x}{\lambda} \right\rangle^{-(d+1)}
        \end{equation*}
        for all $\lambda \geq 1$.
    \end{prop}

    \begin{proof}
        First of all, we remark that the relation \eqref{eq:KernelLambda} holds for all cut-off functions $\theta$, so it is also possible to adjust the size of the cut-off $\theta$ according to the value of $\lambda$ by scaling. By setting $\theta_\lambda(x) = \theta(x / \lambda)$, we see that the kernel $\Gamma_\lambda$ is given by
        \begin{equation*}
            \Gamma_\lambda = \nabla^2 \psi_\lambda * \nabla (\theta_\lambda E) + \psi_\lambda * \nabla^3 \big( (1 - \theta_\lambda) E \big).
        \end{equation*}

        \medskip
        
        Let us first focus on the convolution product $\nabla^2 \psi_\lambda * \nabla (\theta_\lambda E)$. We start by remarking that, because $\psi \in \mc S$ is a Schwartz function, it must decay faster than any polynomial, along with all of its derivatives. Therefore, for any $\mu > 0$, we write the inequality
        \begin{equation*}
            |\psi(y)| + |\nabla^2 \psi(y)| \leq C(\mu) \langle y \rangle^{- \mu}.
        \end{equation*}
        The exponent $\mu$ will be taken as large as necessary, and we will fix its precise value later. By plugging this estimate in the convolution product, we obtain
        \begin{equation*}
            \big| \nabla^2 \psi_\lambda * \nabla (\theta_\lambda E)(x) \big| \lesssim \frac{1}{\lambda^{d+2}} \int \left\langle \frac{y}{\lambda} \right\rangle^{- \mu} \frac{\mathds{1}_{|x - y| \leq 2 \lambda}}{|x - y|^{d-1}} \dy.
        \end{equation*}
        Notice that this estimate is true independently of whether $d=2$ or not. As is now standard procedure, we cut the integral into two parts, by using the decomposition of the space $\R^d = \{ |y| \leq \frac{1}{3}|x| \} \sqcup \{ |y| > \frac{1}{3}|x| \}$. For the integral on the first part $|y| \leq \frac{1}{3}|x|$, we have $\frac{2}{3}|x| \leq |x - y| \leq \frac{4}{3}|x|$, and so we have a first series of inequalities:
        \begin{equation*}
            \frac{1}{\lambda^{d+2}} \int_{|y| \leq \frac{1}{3}|x|} \left\langle \frac{y}{\lambda} \right\rangle^{- \mu} \frac{\mathds{1}_{|x - y| \leq \lambda}}{|x - y|^{d-1}} \dy \lesssim \frac{\mathds{1}_{|x| \leq 3\lambda}}{\lambda^{d+2} |x|^{d-1}} \int_{|y| \leq \frac{1}{3}|x|} \left\langle \frac{y}{\lambda} \right\rangle^{- \mu} \dy.
        \end{equation*}
        By using the fact that the function $\langle y / \lambda \rangle^{- \mu}$ is bounded by $1$, we may estimate the integral by the volume of the integration domain, namely $C(d) |x|^d$. Consequently,
        \begin{equation}\label{eq:KernelEstimateEQ1}
            \frac{1}{\lambda^{d+2}} \int_{|y| \leq \frac{1}{3}|x|} \left\langle \frac{y}{\lambda} \right\rangle^{- \mu} \frac{\mathds{1}_{|x - y| \leq \lambda}}{|x - y|^{d-1}} \dy \lesssim \frac{|x|}{\lambda^{d+2}} \mathds{1}_{|x| \leq 3\lambda} \lesssim \frac{1}{\lambda^{d+1}} \mathds{1}_{|x| \leq 3\lambda}.
        \end{equation}
        For the other part of the integral $|y| > \frac{1}{3}|x|$, we see that the quantity $\langle y / \lambda \rangle$ is bounded from below by $C\langle x / \lambda \rangle$, which leads to 
        \begin{equation}\label{eq:KernelEstimateEQ2}
            \begin{split}
                \frac{1}{\lambda^{d+2}} \int_{|y| > \frac{1}{3}|x|} \left\langle \frac{y}{\lambda} \right\rangle^{- \mu} \frac{\mathds{1}_{|x - y| \leq 2\lambda}}{|x - y|^{d-1}} \dy & \lesssim \left\langle \frac{x}{\lambda} \right\rangle^{- \mu} \frac{1}{\lambda^{d+2}} \int \frac{\mathds{1}_{|x-y| \leq 2 \lambda}}{|x - y|^{d-1}} \dy \\
                & \lesssim \frac{1}{\lambda^{d+1}} \left\langle \frac{x}{\lambda} \right\rangle^{- \mu}.
            \end{split}
        \end{equation}

        \medskip

        We now center our attention on the second convolution product $\psi_\lambda * \nabla^3 \big( (1 - \theta_\lambda)E \big)$. By using again the Schwartz estimate for $\psi$, we obtain the inequality
        \begin{equation*}
            \big| \psi_\lambda * \nabla^3 \big( (1 - \theta_\lambda)E \big)(x) \big| \lesssim \frac{1}{\lambda^d} \int \left\langle \frac{y}{\lambda} \right\rangle^{- \mu} \frac{\mathds{1}_{|x - y| \geq \lambda}}{|x - y|^{d+2}} \dy.
        \end{equation*}
        This estimate holds independently of whether $d=2$ or not. Just as above, we split this integral in two parts $\R^d = \{ |y| \leq \frac{1}{3}|x| \} \sqcup \{ |y| > \frac{1}{3}|x| \}$, the first one giving the bound
        \begin{equation*}
            \frac{1}{\lambda^d} \int_{|y| \leq \frac{1}{3}|x|} \left\langle \frac{y}{\lambda} \right\rangle^{- \mu} \frac{\mathds{1}_{|x - y| \geq \lambda}}{|x - y|^{d+2}} \dy \lesssim \frac{\mathds{1}_{|x| \geq 3\lambda/4}}{|x|^{d+1}} \int_{|y| \leq \frac{1}{3}|x|} \left\langle \frac{y}{\lambda} \right\rangle^{- \mu} \frac{\dy}{\lambda^d}.
        \end{equation*}
        In order to obtain an upper bound that has sufficient integrability, we should avoid bounding the integral by a power of $|x|$ as we did above. Instead, we use the change of variables $\lambda y \leftarrow y$ and obtain the upper bound
        \begin{equation}\label{eq:KernelEstimateEQ3}
            \frac{1}{\lambda^d} \int_{|y| \leq \frac{1}{3}|x|} \left\langle \frac{y}{\lambda} \right\rangle^{- \mu} \frac{\mathds{1}_{|x - y| \geq \lambda}}{|x - y|^{d+2}} \dy \lesssim \frac{\mathds{1}_{|x| \geq 3 \lambda /4}}{|x|^{d+1}}.
        \end{equation}
        Of course, just like the upper bounds obtained above, it should be noted that this one decays as $O(\lambda^{-d-1})$ when $\lambda \rightarrow \infty$. Concerning the other half of the integral, on the integration domain $|y| > \frac{1}{3}|x|$, we also use the fact that $\langle y / \lambda \rangle$ is bounded from below by $C \langle x / \lambda \rangle$, which leads to the inequality
        \begin{equation}\label{eq:KernelEstimateEQ4}
            \begin{split}
                \frac{1}{\lambda^d} \int_{|y| > \frac{1}{3}|x|} \left\langle \frac{y}{\lambda} \right\rangle^{- \mu} \frac{\mathds{1}_{|x - y| \geq \lambda}}{|x - y|^{d+2}} \dy & \lesssim \left\langle \frac{x}{\lambda} \right\rangle^{- \mu} \frac{1}{\lambda^d} \int \frac{\mathds{1}_{|x-y| \geq \lambda}}{|x - y|^{d+1}} \dy \\
                & \lesssim \frac{1}{\lambda^{d+1}} \left\langle \frac{x}{\lambda} \right\rangle^{- \mu}.
            \end{split}
        \end{equation}

        \medskip

        Finally, we assemble all four estimates \eqref{eq:KernelEstimateEQ1}, \eqref{eq:KernelEstimateEQ2}, \eqref{eq:KernelEstimateEQ3} and \eqref{eq:KernelEstimateEQ4}. They provide the bound
        \begin{equation*}
            |\Gamma_\lambda(x)| \lesssim \frac{\mathds{1}_{|x| \leq 3\lambda}}{\lambda^{d+1}} + \frac{\mathds{1}_{|x| \geq 3 \lambda /4}}{|x|^{d+1}} + \frac{1}{\lambda^{d+1}} \left\langle \frac{x}{\lambda} \right\rangle^{- \mu}.
        \end{equation*}
        We see that the first two terms in the upper bound have similar structure, and have supports that only intersect when $|x| \approx \lambda$. This allows us to write the inequality
        \begin{equation*}
            \begin{split}
                \frac{\mathds{1}_{|x| \leq 3\lambda}}{\lambda^{d+1}} + \frac{\mathds{1}_{|x| \geq 3 \lambda /4}}{|x|^{d+1}} & \lesssim \frac{1}{\big( |x| + \lambda \big)^{d+1}} = \frac{1}{\lambda^{d+1}} \left( 1 + \frac{|x|}{\lambda} \right)^{-(d+1)} \\
                & \lesssim \frac{1}{\lambda^{d+1}} \left\langle \frac{x}{\lambda} \right\rangle^{-(d+1)}.
            \end{split}
        \end{equation*}
        By choosing $\mu$ such that $\mu > d+1$, we end the proof of the Proposition.
    \end{proof}
    
    \textbf{STEP 2: convolution product estimate.} In this second part of the proof, we use the kernel estimates we have just derived to prove that \eqref{eq:PressureConvergence} holds at any given time. More precisely, we show the estimates which are enclosed in the following proposition.

    \begin{prop}\label{p:SphDecay}
        Consider a function $f \in L^2_\alpha (\R^d ; \R^d)$. Then the convolution product $\Gamma_\lambda * (f \otimes f)$ is equal to a sum $I_1 + I_2$ where
        \begin{equation}\label{eq:I1Decay}
            \| I_1 \|_{L^\infty} \lesssim \frac{1}{\lambda^{1 - 2 \alpha}} \| f \|_{L^2_\alpha}^2
        \end{equation}
        and
        \begin{equation}\label{eq:I2Decay}
            \| I_2 \|_{L^\infty_{2\alpha}} \lesssim \frac{1}{\lambda^{d+1}} \| f \|_{L^2_\alpha}^2.
        \end{equation}
        In particular, we have $\nabla \bar{\Pi}(f \otimes f) \in \mc S'_h$.
    \end{prop}

\begin{rmk}
    It should be noted that the estimates above show that the quantity $\chi(\lambda D)\nabla \bar{\Pi}(f \otimes f)$ converges to zero as $\lambda \rightarrow \infty$ in the strong topology of $L^\infty_{2 \alpha}$ (and not only in the sense of distributions). However, it is not true that the convergence holds in $L^\infty$. This highlights the importance of choosing the correct topology when defining the space $\mc S'_h$ (see Remark \ref{r:Chemin}). If we had defined $\mc S'_h$ like it is in the classical textbook \cite{BCD}, as the space of all $f\in \mc S'$ such that
    \begin{equation*}
        \chi(\lambda D)f \tend_{\lambda \rightarrow \infty} 0 \qquad \text{in } L^\infty,
    \end{equation*}
    then Theorem \ref{t:RepresentationPressure} would not hold. The definition of \cite{BCD} is adapted for PDEs in spaces that embed in $L^\infty$, but not for spaces of unbounded functions.
\end{rmk}

    \begin{proof}
        By invoking the kernel estimate of Proposition \ref{p:KernelEstimate}, we bound the convolution product by
        \begin{equation*}
            |\Gamma_\lambda * (f \otimes f) (x)| \lesssim \frac{1}{\lambda} \int |f(y)|^2 \left\langle \frac{x-y}{\lambda} \right\rangle^{-(d+1)} \frac{\dy}{\lambda^{d}},
        \end{equation*}
        and we again resort to the trick of decomposing the integral in two parts. This time, we write the space as a partition $\R^d = \{ |x| \leq \frac{1}{3}|y| \} \sqcup \{ |x| > \frac{1}{3}|y| \}$. In the first of these integrals $|y| \leq \frac{1}{3}|x|$, which we call $J_1(x)$ (and corresponds to $I_1$), we have $\frac{2}{3}|x| \leq |x - y| \leq \frac{4}{3}|x|$, and so
        \begin{equation*}
            \begin{split}
                J_1(x) & := \frac{1}{\lambda} \int_{|x| \leq \frac{1}{3}|y|} |f(y)|^2 \left\langle \frac{x-y}{\lambda} \right\rangle^{-(d+1)} \frac{\dy}{\lambda^{d}} \\
                & \lesssim \frac{1}{\lambda} \int |f(y)|^2 \frac{\dy}{\lambda^d \langle y / \lambda \rangle^{d+1}} \\
                & \lesssim \frac{1}{\lambda} \int \big| f(\lambda y) \big|^2 \frac{\dy}{\langle y \rangle^{d+1}}.
            \end{split}
        \end{equation*}
        In order to evaluate this integral using the $L^2_\alpha$ norms of $f$, we resort to a dyadic decomposition of the space that is very close to the one presented in \eqref{eq:dydadicSplittingR3Existence} above. We set
        \begin{equation*}
            \R^d := \bigsqcup_{k=0}^\infty A_k := \big\{ |y| < 2 \big\} \sqcup \bigsqcup_{k=1}^\infty \big\{ 2^k \leq |y| < 2^{k+1} \big\}.
        \end{equation*}
        Then, by decomposing the integral above in a sum of integrals over the $A_k$, changing variables, and using the definition of the $L^2_\alpha$ norm, we have:
        \begin{equation*}
            \begin{split}
                J_1(x) & \lesssim \frac{1}{\lambda} \sum_{k=0}^\infty \int_{A_k} \big| f(\lambda y) \big|^2 \frac{\dy}{\langle y \rangle^{d+1}} \\
                & \lesssim \frac{1}{\lambda} \sum_{k=0}^\infty 2^{-k(d+1)} \int_{\lambda A_k} | f(y) |^2 \frac{\dy}{\lambda^d} \\
                & \lesssim \frac{1}{\lambda} \| f \|_{L^2_\alpha}^2 \sum_{k=0}^\infty 2^{-k(d+1)} \big( \lambda 2^k \big)^{d + 2 \alpha} \frac{1}{\lambda}
            \end{split}
        \end{equation*}
        and therefore, we establish the inequality
        \begin{equation*}
            \| J_1 \|_{L^\infty} \lesssim \frac{1}{\lambda^{1 - 2 \alpha}} \| f \|_{L^2_\alpha}^2,
        \end{equation*}
        from which \eqref{eq:I1Decay} follows.

        \medskip

        On the other hand, in the integral on $|x| > \frac{1}{3}|y|$, the integration variable $y$ is bounded by $3|x|$. Let us fix an $R \geq 1$ and assume that $|x| \leq R$. By writing 
        \begin{equation*}
            \begin{split}
                J_2(x) & = \frac{1}{\lambda^{d+1}} \int |f(y)|^2 \mathds{1}_{|y| \leq 3|x|} \left\langle \frac{x-y}{\lambda} \right\rangle^{-(d+1)} \dy \\
                & \lesssim \frac{1}{\lambda^{d+1}} \int |f(y)|^2 \mathds{1}_{|y| \leq 3R} \left\langle \frac{x-y}{\lambda} \right\rangle^{-(d+1)} \dy \\
                & \lesssim \frac{1}{\lambda^{d+1}} \int_{B_{3R}} \big| f(y) \big|^2 \dy \\
                & \lesssim \frac{R^{d + 2 \alpha}}{\lambda^{d+1}} \| f \|_{L^2_\alpha}^2.
            \end{split}
        \end{equation*}
        By definition of the $L^\infty_{2 \alpha}$ norm this proves the estimate.

    \end{proof}

    In particular, Proposition \ref{p:SphDecay} shows that $\nabla \Pi(f \otimes f) \in \mc S'_h$ whenever $f \in L^2_\alpha$. The same remark holds for the quantity $\D(f \otimes f)$. More precisely, we have
    \begin{equation*}
        \psi_\lambda * \D(f \otimes f) = \nabla \psi_\lambda * (f \otimes f),
    \end{equation*}
    and as the kernel 
    \begin{equation}\label{eq:SpHDiv}
        \nabla \psi_\lambda (y) = \frac{1}{\lambda^{d+1}} \nabla \psi \left( \frac{y}{\lambda} \right)
    \end{equation}
    also fulfills the estimates of Proposition \ref{p:KernelEstimate}, the conclusion of Proposition \ref{p:SphDecay} also hold: $\D(f \otimes f) \in \mc S'_h$.

    \medskip

    \textbf{STEP 3: end of the proof.} We start by seeking a reformulation of the Euler equations which makes appear the pressure force operator $\nabla \bar{\Pi} (u \otimes u)$ and a polynomial remainder. For convenience, we extend the solution to negative times $\wtilde{u} \in L^\infty(]- \infty, T[ ; L^2_\alpha)$ by setting $\wtilde{u}(t) = 0$ for all $t < 0$. We also extend the pressure force in the same way: we define a distribution $\wtilde{F_\pi} \in \mc D'(]- \infty , T[\times \R^d)$ by setting, for all $\phi \in \mc D (]- \infty ; T[ \times \R^d)$, 
    \begin{equation*}
        \langle \wtilde{F_\pi}, \phi \rangle = \langle F_\pi, \mathds{1}_{t \geq 0} \phi \rangle.
    \end{equation*}
    This is possible because the restriction to non-negative times satisfies $\mathds{1}_{t \geq 0} \phi \in \mc D ([0, T[ \times \R^d])$ (see Remark \ref{r:IrregularPressure}). For simplicity of notation, we continue to note $F_\pi$ and $u$ these extensions. Then, the extensions satisfy the equation
    \begin{equation}\label{eq:WeakFormExtended}
        \partial_t u + \D(u \otimes u) + F_\pi = \delta_0(t) \otimes u_0(x),
    \end{equation}
    where the tensor product $\delta_0(t) \otimes u_0(x)$ is defined by the relation
    \begin{equation*}
        \forall \phi \in \mc D (]- \infty, T[ \times \R^d), \qquad \langle \delta_0 \otimes u_0 , \phi \rangle := \int u_0(x) \phi(0, x) \dx.
    \end{equation*}

    \medskip

    Then, by taking the divergence and the gradient of equation \eqref{eq:WeakFormExtended}, we see that the pressure force solves the elliptic equation
    \begin{equation*}
        - \Delta F_\pi = \nabla \partial_j \partial_k (u_j u_k),
    \end{equation*}
    and this holds in $\mc D'(\R^d)$ for all times $t < 0$, as the velocity is $C^0(L^2_\alpha)$, and so $u \otimes u \in C^0(\mc S')$. In particular, the pressure force $F_\pi$ and $\nabla \bar{\Pi} (u \otimes u)$ are both $\mc S'$ solutions of the same Poisson equation at every time $t < T$, and they must therefore differ by a harmonic polynomial $Q(t) \in \R[X]$. We deduce that
    \begin{equation}\label{eq:ProjectodWithQ}
        \partial_t u + \D (u \otimes u) + \nabla \bar{\Pi}(u \otimes u) + Q = \delta_0(t) \otimes u_0(x).
    \end{equation}
    Now, the whole point of the proof is to examine the condition upon which we will have $Q = 0$. 

    \medskip

    Before proving the equivalence of the statements in Theorem \ref{t:RepresentationPressure}, we write a weak form of equation \eqref{eq:ProjectodWithQ} and reorganize it in order to let the quantity $u(t) - u(0)$ appear. Consider a test function $\Phi \in \mc D (]- \infty, T[ \times \R^d ; \R^d)$. By applying the low-frequency cut-off operator $\chi(\lambda D)$ to equation \eqref{eq:ProjectodWithQ}, multiplying by $\Phi$ and integrating with respect to time and space, we have:
    \begin{multline}\label{eq:WeakIntegralsEQ1}
        - \int_{- \infty}^T \int \partial_t \Phi \cdot \chi(\lambda D) u \dx \dt + \int_{- \infty}^T   \big\langle \chi(\lambda D) \D (u \otimes u), \Phi(t) \big\rangle_x \dt \\
        + \int_{- \infty}^T \big\langle \chi(\lambda D) \nabla \bar{\Pi} (u \otimes u), \Phi (t) \big\rangle_x \dt + \langle \chi(\lambda D) Q, \Phi \rangle_{t, x} = \phi(0) \big\langle \chi(\lambda D) u_0, \Phi(0) \rangle_x,
    \end{multline}
    where, in the above, the brackets $\langle \, \cdot \, , \, \cdot \, \rangle_x$ refer to the space duality $\mc D'(\R^d) \times \mc D (\R^d)$ only, whereas the brackets $\langle \, \cdot \, , \, \cdot \, \rangle_{t, x}$ refer to the full time-space duality on $]- \infty, T[ \times \R^d$. On the one hand, Proposition \ref{p:SphDecay} shows that the quantity $\nabla \bar{\Pi}(u \otimes u)$ is an element of $\mc S'_h$. In other words (see Definition \ref{d:Chemin}), we have the convergence
    \begin{equation*}
        \chi(\lambda D) \nabla \bar{\Pi}(u \otimes u) \tend_{\lambda \rightarrow \infty} 0 \qquad \text{in } \mc S',
    \end{equation*}
    and the same property holds for $\D(u \otimes u)$ by \eqref{eq:SpHDiv} and the remarks following that equation. Applying this to the test function $\Phi$, we see that for every time $t \in ]- \infty, T[$ we must have
    \begin{equation*}
        \big\langle \chi(\lambda D)  \big( \D (u \otimes u) + \nabla \bar{\Pi}(u \otimes u) \big), \Phi(t) \big\rangle_x \tend 0 \qquad \text{as } \lambda \rightarrow \infty,
    \end{equation*}
    and so, by dominated convergence, we also have
    \begin{equation*}
        \int_{- \infty}^T \big\langle \chi(\lambda D) \big( \D (u \otimes u) + \nabla \bar{\Pi}(u \otimes u) \big), \Phi(t) \big\rangle_x \dt \tend 0 \qquad \text{as } \lambda \rightarrow \infty.
    \end{equation*}
    On the other hand, we notice that the first integral in \eqref{eq:WeakIntegralsEQ1} can be rewritten in order to make the difference $u(t) - u(0)$ appear. Here it is more convenient to take $\Phi$ of the tensor product form $\Phi(t,x) = \phi(t) \varphi(x)$ with
    \begin{equation*}
        \phi \in \mc D (]- \infty, T[) \qquad \text{and} \qquad \varphi \in \mc D(\R^d).
    \end{equation*}
    Because $\phi(0) = -\int_{\R_+} \phi(t) \dt$, we have
    \begin{equation*}
        \begin{split}
            \int_{- \infty}^T \phi'(t) \big\langle \chi(\lambda D) u(t), \varphi \big\rangle_x \dt & = \int_0^T \phi'(t) \big\langle \chi(\lambda D) u(t), \varphi \big\rangle_x \dt \\
            & = \int_0^T \phi'(t) \big\langle \chi(\lambda D) \big( u(t) - u(0) \big), \varphi \big\rangle_x \dt - \phi(0) \big\langle \chi(\lambda D) u(0), \varphi \big\rangle_x.
        \end{split}
    \end{equation*}
    Now, also note that $\chi(\lambda D)Q = Q$, because the Fourier transform of $Q(t)$ is supported in $\{0\}$, as $Q(t) \in \R[X]$. By putting everything together in equation \eqref{eq:WeakIntegralsEQ1}, we find that, as $\lambda \rightarrow \infty$,
    \begin{equation*}
        \big\langle Q, \phi \varphi \big\rangle_{t, x} = \phi(0) \big\langle \chi (\lambda D) \big( u_0 - u(0) \big), \varphi \big\rangle_x + \int_0^T \phi'(t) \big\langle \chi(\lambda D) \big( u(t) - u(0) \big), \varphi \big\rangle_x \dt + o(1).
    \end{equation*}
    Now, the difference $u_0 - u(0)$ is a divergence-free function and is orthogonal to all divergence-free functions in the $\mc D' \times \mc D$ duality, since $u$ is a weak solution of the Euler equations with initial datum $u_0$ (by assumption) \textsl{and} with initial datum $u(0)$ (because $u \in C^0_T(L^2_\alpha)$). We refer to Remark \ref{r:nonEquivalence} for additional comments on this particular phenomenon. This function must therefore be a polynomial $u_0 - u(0) \in \R[X]$, and is of degree at most zero since it also lies in $L^2_\alpha$. We deduce that $\chi (\lambda D) \big( u_0 - u(0) \big) = u_0 - u(0)$ and the previous equation becomes
    \begin{equation}\label{eq:LastSummaryLastTheorem}
        \big\langle Q, \phi \varphi \big\rangle_{t, x} = \phi(0) \big\langle u_0 - u(0), \varphi \big\rangle_x + \lim_{\lambda \rightarrow \infty} \int_0^T \phi'(t) \big\langle \chi(\lambda D) \big( u(t) - u(0) \big), \varphi \big\rangle_x \dt.
    \end{equation}
    
    \medskip

    We are ready to prove the equivalence of the assertions in Theorem \ref{t:RepresentationPressure}. First assume \textit{(ii)}, that  $u(t) - u(0) \in \mc S'_h$ for all times $0 \leq t < T$ and that $u_0 = u(0)$. Dominated convergence provides the limit
    \begin{equation}\label{eq:VelocityDifferenceDominatedConvergence}
        \int_0^T \phi'(t) \big\langle \chi(\lambda D) \big( u(t) - u(0) \big), \varphi \big\rangle_x \dt \tend 0 \qquad \text{as } \lambda \rightarrow \infty.
    \end{equation}
    Consequently, equation \eqref{eq:LastSummaryLastTheorem} shows that $Q = 0$ in $\mc D' (]- \infty, T[ \times \R^d)$, so that $F_\pi$ must be equal to $\nabla \bar{\Pi}(u \otimes u)$ in that space. This proves \textit{(i)}, and it suffices to remark that $\nabla \bar{\Pi}(u \otimes u) \in C^0([0, T] ; \mc S')$ lies in $\mc S'_h$ to make sure that \textit{(iii)} also holds.

    \medskip

    Suppose that \textit{(i)} is true. Then $Q = 0$ in $\mc D'(](- \infty, T[ \times \R^d)$. By picking $\phi(t)$ that is supported away from $t = 0$, we see from \eqref{eq:LastSummaryLastTheorem} that $u(t) - u(0) \in \mc S'_h$ for any $t > 0$, and this trivially extends to $t \geq 0$. And now, by taking $\phi$ and $\varphi$ arbitrary, we see that \eqref{eq:VelocityDifferenceDominatedConvergence} holds, and so we must have $u_0 = u(0)$, which proves point \textit{(ii)}. Point \textit{(iii)} also holds since $F_\pi = \nabla \bar{\Pi}(u \otimes u) \in C^0([0, T] ; \mc S')$ lies in $\mc S'_h$.

    \medskip

    Finally, assume that \textit{(iii)} holds. Since we know that $\nabla \bar{\Pi}(u \otimes u) \in \mc S'_h$ for every time $t \geq 0$ (from Proposition \ref{p:SphDecay}), this means that $Q = F_\pi - \nabla \bar{\Pi} (u \otimes u) \in C^0([0, T[ ; \mc S')$ and that $Q(t) \in \mc S'_h$ for every time. As a consequence point \textit{(i)} holds and the above shows that point \textit{(ii)} follows. This completes the proof.
\end{proof}

\end{document}